\documentclass[11pt,english,a4paper]{smfart}
\usepackage[english]{babel}
 \usepackage{enumerate}
\usepackage{amsmath,amssymb,bm,amsfonts,bbold}
\usepackage{bbm}
\usepackage[all]{xy}
\entrymodifiers={+!!<0pt,\fontdimen22\textfont2>}

\usepackage{mathrsfs}
 \usepackage{enumitem}
\usepackage{fontenc}
\usepackage{inputenc}

\newtheorem{theorem}{Theorem}[section]

\newtheorem{corollary}[theorem]{Corollary}
\newtheorem{proposition}[theorem]{Proposition}

\newcommand{\fl}[2]{
\xymatrix@C15pt{#1\ar[r]&#2}}
\newcommand{\flcourte}[2]{
\xymatrix@C12pt{#1\ar[r]&#2}}

{\theoremstyle{definition}}

{\theoremstyle{definition}\newtheorem{example}[theorem]{Example}}

\theoremstyle{definition}
\newtheorem{definition}[theorem]{Definition}
\newtheorem{question}[theorem]{Question}

\newtheorem{claim}[theorem]{Claim}

{\theoremstyle{definition}\newtheorem{remark}[theorem]{Remark}}

\def\T{\ensuremath{\mathbb T}}
\def\R{\ensuremath{\mathbb R}}
\def\Z{\ensuremath{\mathbb Z}}

\def\C{\ensuremath{\mathbb C}}
\def\Q{\ensuremath{\mathbb Q}}
\def\N{\ensuremath{\mathbb N}}

\newcommand{\wh}[1]{\widehat{#1}}

\newcommand{\ds}{\displaystyle}
\newcommand{\ba}[1]{\overline{#1}}

\newcommand{\ti}[1]{\widetilde{#1}}

\newcommand{\ka}{Kazhdan}

\newcommand{\qq}{Q}

\newcommand{\nq}{n\in\qq}

\newcommand{\nk}{n_{k}}
\newcommand{\nkp}[1]{(n_{k})_{k\ge #1}}
\newcommand{\mk}{m_{k}}
\newcommand{\mc}{\wh{\mu }}
\newcommand{\dpi}[1]{\delta _{\{#1\}}}

\newcommand{\ul}{\liminf}
\DeclareMathOperator{\Kc}{Kaz}
\newcommand{\Kct}{\widetilde{\Kc}}

\author{Catalin Badea}
\address{Laboratoire Paul Painlev\'e, UMR 8524\\
Universit\'{e} de Lille\\
Cit\'e Scientifique, B\^atiment M2\\
59655 Villeneuve d'Ascq Cedex\\
France}
\email{catalin.badea@univ-lille.fr}
\author{Sophie Grivaux}
\address{CNRS, Laboratoire Paul Painlev\'e, UMR 8524\\
Universit\'{e} de Lille\\
Cit\'e Scientifique, B\^atiment M2\\
59655 Villeneuve d'Ascq Cedex\\
France}
\email{sophie.grivaux@univ-lille.fr}
\title[Continuous probability measures with large Fourier coefficients]{Kazhdan constants, continuous probability measures with large Fourier coefficients and rigidity sequences}
\thanks{This work was supported in part by
 the project FRONT of the French
National Research Agency (grant ANR-17-CE40-0021) and by the Labex CEMPI (ANR-11-LABX-0007-01).
}

\thanks{We are grateful to \'Etienne Matheron for pointing out a simplification of our original proof of Corollary \ref{Cor2}, and to \'Etienne Matheron, Martine Queff\'elec, Jean-Paul Thouvenot and Benjy Weiss for several interesting discussions.}

\begin{document}
\begin{abstract}
Exploiting a construction of rigidity sequences for weakly mixing dynamical systems by Fayad and Thouvenot, we show that for every integers $p_{1},\dots,p_{r}$ there exists a continuous probability measure $\mu $ on the unit circle $\T$ such that 
\[
\inf_{k_{1}\ge 0,\dots,k_{r}\ge 0}|\wh{\mu }(p_{1}^{k_{1}}\dots p_{r}^{k_{r}})|>0.
\]
This results applies in particular to the Furstenberg set $F=\{2^{k}3^{k'}\,;\,k\ge 0,\ k'\ge 0\}$, and disproves a 1988 conjecture of Lyons inspired by Furstenberg's famous $\times 2$-$\times 3$ conjecture. We also estimate the modified Kazhdan constant of $F$ and obtain general results on rigidity sequences which allow us to retrieve essentially all known examples of such sequences.
\end{abstract}

\dedicatory{
\bigskip
To the memory of Jean-Pierre Kahane (1926-2017)
\bigskip\bigskip\bigskip}

\maketitle

\section{Introduction}\label{Intro}
Denote by $\T$ the unit circle $\T=\{\lambda \in\C\,;\,|\lambda |=1\}$, by $\mathcal{M}(\T)$ the set of (finite) complex Borel measures on $\T$ and by $\mathcal{P}(\T)$ the set of Borel probability measures on $\T$. The Fourier coefficients of $\mu\in \mathcal{M}(\T)$ are defined here as 
$$\hat{\mu}(n) = \int_{\T} \lambda ^n \, d\mu(\lambda) .$$
A measure $\mu \in\mathcal{P}(\T)$ is said to be \emph{continuous}, or \emph{atomless}, if $\mu (\{\lambda \})=0$ for every $\lambda \in\T$.  We denote the set of continuous probability measures on $\T$ by $\mathcal{P}_{c}(\T)$. According to a theorem of Wiener and the Koopman-von Neumann lemma, $\mu$ is continuous if and only if $\hat{\mu}(n)$ tends to zero as $n$ tends to infinity along a sequence in $\N$ of density one. For every $\mu \in\mathcal{P}(\T)$, we define $\ti{\mu }$ by setting $\ti{\mu }(A)=\mu (A^{c})$ for every Borel set $A\subseteq \T$, with $A^{c}=\{\ba{\lambda }\,;\,\lambda \in A\}$. Then $\nu:=\mu *\ti{\mu }$ has the property that $\hat{\nu}(n)=|\hat{\mu}(n)|^2\ge 0$ for every $n\in\Z$, and $\nu$ belongs to $\mathcal{P}_{c}(\T)$ as soon as $\mu$ does.

\subsection*{A conjecture of Russell Lyons} 
Our aim in this paper is to study some non-lacunary sets of positive integers from a Fourier analysis point of view, and to construct some probability measures which have large Fourier coefficients on such sets. In particular, we disprove a 1988 conjecture of Lyons \cite{L}, called there Conjecture (C4), which reads as follows:
\par\medskip
\begin{flushright}
 \begin{minipage}{14cm}Lyons' Conjecture (C4): {\emph {If $S$ is a non-lacunary semigroup of integers, and if $\mu \in\mathcal{P}_{c}(\T)$, there exists an infinite sequence $(n_{k})_{k\ge 1}$ of elements of $S$ such that ${\wh{\mu }(n_{k})}\rightarrow{0}$ as ${k}\rightarrow{+\infty}$.}}
 \end{minipage}
\end{flushright}
\par\medskip
This conjecture of Lyons is inspired by Furstenberg's famous conjecture concerning simultaneously invariant probability measures for two commuting automorphisms of the unit circle $\T$,
${T_{p}:\lambda \longmapsto\lambda ^{p}}$ and ${T_{q}:\lambda \longmapsto\lambda ^{q}}$,  when $p$ and $q$ are two multiplicatively independent integers (i.e.\ $p$ and $q$ are not both powers of the same integer). In this setting, Furstenberg's conjecture states that the only continuous probability measure on $\T$ invariant by both $T_{p}$ and $T_{q}$ is the Lebesgue measure on $\T$. Furstenberg himself proved \cite{F} that if $S$ is any non-lacunary semigroup of integers (i.e.\ if $S$ is not contained in any semigroup of the form $\{a^{n}\,;\,n\ge 0\}$, $a\ge 2$), then the only infinite closed $S$-invariant subset of $\T$ is $\T$ itself. See \cite{Bos} for an elementary proof of this result and the references mentioned in \cite[Chapter 2]{Bug} for several extensions. Since $S=\{p^{k}q^{k'}\,;\,k,\,k'\ge 0\}$ is a non-lacunary semigroup whenever $p$ and $q$ are multiplicatively independent, the only infinite closed subset of $\T$ which is simultaneously $T_{p}$-invariant and $T_{q}$-invariant is $\T$. Starting with the work of Lyons in \cite{L}, Furstenberg's conjecture has given rise to an impressive amount of related questions and results, concerning in particular the dynamics of commuting group automorphisms. We refer the reader to the papers \cite{Ru}, \cite{BLMV}, \cite{EF} or \cite{Hoch-autre} for example, as well as to the texts \cite{Lind1}, \cite{Hoch} or \cite{Ve} for  surveys of results related to this conjecture, as well as for perspectives.
\par\smallskip
As written in \cite{L}, conjecture (C4) is a natural version of Furstenberg's conjecture about measures, but not involving invariance. If (C4) were true, it would imply an affirmative answer to the Furstenberg conjecture (if $\mu\in\mathcal{P}_c(\T)$ is $T_{p}$- and $T_{q}$-invariant, applying (C4) to each of the measures $\mu_j:=T_j(\mu)$, $j\in\Z\setminus\{0\}$, yields that $\hat{\mu}(j)=0$ for every $j\in\Z\setminus\{0\}$). 

\subsection*{Kazhdan sets and modified \ka\ constants} It turns out that Lyons' conjecture is related to an important property of subsets of $\Z$, namely that of being or not a \emph{\ka\ subset} of $\Z$.  
 \ka\ subsets $Q$ of a second-countable topological group $G$ are those for which there exists $\varepsilon >0$ such that any strongly continuous representation $\pi $ of $G$ on a complex separable Hilbert space $H$ admitting  a vector $x\in H$ with $||x||=1$ which is $\varepsilon $-invariant on $Q$ (i.e.\ $\sup_{g\in Q}||\pi (g)x-x||<\varepsilon $) has a $G$-invariant vector. Such an $\varepsilon $ is called a \emph{\ka\ constant for} $Q$, and the supremum of all $\varepsilon $'s with this property is {the \ka\ constant of} $Q$. 
Groups with Property (T), also called \ka\ groups, are those which admit a compact \ka\ set. See the book \cite{BdHV} for more on Property (T) and its numerous important applications. 
\par\smallskip
As suggested in \cite[Sec. 7.12]{BdHV}, it is of interest to study \ka\ sets in groups which do not have Property (T), such as locally compact abelian groups, Heisenberg groups, etc. See \cite{BG1} and also \cite{Ch} for a study of such problems. In the case of the group $\Z$, the definition above is equivalent to the following one:

\begin{definition}(Kazhdan sets and constants)
A subset $Q\subset \Z$ is said to be a \emph{Kazhdan set} if there exists $\varepsilon >0$ such that any unitary operator $U$ acting on a complex separable Hilbert space $H$ satisfies the following property: if there exists a vector $x\in H$ with $||x||=1$ such that $\sup_{\nq}||U^{n}x-x||<\varepsilon $, then there exists a non-zero vector $y\in H$ such that $Uy=y$ (i.e.\ $1$ is an eigenvalue of $U$). We will say in this case that $(Q,\varepsilon)$ is a Kazhdan pair. We define the Kazhdan constant of $Q$ as
$$ \Kc(Q) = \inf_{U}\inf_{\|x\|=1}\sup_{q\in Q} \|U^qx-x\|,$$
where the first infimum is taken over all unitary operators $U$ on $H$ without fixed vectors.
\end{definition}
It follows from  \cite[p. 30]{BdHV} that $0 \le \Kc(Q) \le \sqrt{2}$ for every $Q\subseteq \Z$.
\par\smallskip
Several characterizations of \ka\ subsets of $\Z$ were obtained in \cite{BG1} as consequences of results applying to a much wider class of groups; self-contained proofs of these characterizations of \ka\ subsets of $\Z$, involving only classical tools from harmonic analysis, were obtained in the paper \cite{BG2}. 
One of the characterizations of generating \ka\ sets obtained in \cite[Th.\,6.1]{BG1} (see also \cite[Th.\,4.12]{BG2}) runs as follows. Recall that $Q$ is said to be \emph{generating} in the group $\Z$ if the smallest subgroup containing $Q$ is $\Z$ itself.

\begin{theorem}[\cite{BG1}]\label{Th BB}
Let $Q$ be a generating subset of $\Z$. Then $Q$ is a \ka\ subset of $\Z$ if and only if there exists $\varepsilon' \in(0,\sqrt{2}]$ such that $(Q,\varepsilon' )$ is a \emph{modified \ka\ pair}, that is any unitary operator $V$ acting on a complex separable Hilbert space $H$ satisfies the following property: if there exists a vector $x\in H$ with $||x||=1$ such that $\sup_{\nq}||V^{n}x-x||<\varepsilon' $, then $V$ has at least one eigenvalue.
\end{theorem}
We define now the modified Kazhdan constant of $Q$ as 
$$ \Kct (Q) = \inf_{V}\inf_{\|x\|=1}\sup_{q\in Q} \|V^qx-x\|,$$
where the first infimum is taken this time over unitary operators $V$ on $H$ without eigenvalues (that is, with continuous spectra). Therefore
$$ 0\le \Kc(Q) \le \Kct(Q) \le \sqrt{2}$$
and for every $Q\subseteq \Z$, $\Kc(Q)=0$ if and only if $\Kct(Q)=0$ if and only if $Q$ is a non-Kazhdan set. The property of being or not a \ka\ set can also be  expressed in terms of Fourier coefficients of probability measures; see Section \ref{Section 5} for a discussion.

The characterization of \ka\ subsets of $\Z$ obtained by the authors in \cite{BG1} (see also \cite{BG2}) implies that the generating subsets $Q$ 
of $\Z$ which satisfy the property stated in (C4) (namely that there exists for every $\mu \in\mathcal{P}_{c}(\T)$ an infinite sequence $(n_{k})_{k\ge 1}$ of elements of $Q$ such that ${\wh{\mu }(n_{k})}\rightarrow{0}$ as ${k}\rightarrow{+\infty)}$ are exactly the \ka\ subsets of $\Z$ with modified \ka\ constant $\Kct(Q)=\sqrt{2}$.  Since $\sqrt{2}$ is the modified \ka\ constant of $\Z$ seen as a subset of itself, $\sqrt{2}$ is the maximal modified \ka\ constant, and thus (C4) can be reformulated as: every generating non-lacunary semigroup $S$ of integers is a \ka\ subset of $\Z$ with \emph{maximal} modified \ka\ constant $\sqrt{2}$. 
The relationship between Furstenberg $\times 2$-$\times 3$ conjecture and modified Kazhdan constants can be also seen directly from Proposition \ref{prop:coeff} below.

\section{Main results}\label{sect:2}
The first main result of this paper is the following:

\begin{theorem}\label{Th3}
 Let $p_{1},\dots,p_{r}$ be positive distinct integers and set \[E = \{p_{1}^{k_{1}}\dots p_{r}^{k_{r}}\,;\,k_1\ge0,\,\dots ,k_r\ge 0\}.\] There exists a continuous probability measure $\mu $ on $\T$ such that 
\[
\inf_{k_{1}\ge 0,\dots,k_{r}\ge 0}|\wh{\mu }(p_{1}^{k_{1}}\dots p_{r}^{k_{r}})|>0.
\]
Equivalently, 
$$ \Kct(E) < \sqrt{2}.$$
\end{theorem}
It should be noted that, as conjecture (C4) does not involve invariant measures, we do not assume in Theorem \ref{Th3} that the integers $p_j$ are multiplicatively independent. Notice also that the statement of Theorem \ref{Th3} is well-known in the lacunary case: if $r=1$ it suffices to consider the classical Riesz product associated to the sequence $(p^{k})_{k\ge 0}$. In the non-lacunary case, Theorem \ref{Th3} disproves Conjecture (C4), as well as the related conjectures (C5) and (C6) of \cite{L} (which are both stronger than (C4)). 
It applies in particular to the Furstenberg set $F=\{2^k3^{k'}\textrm{ ; }k,k'\ge 0\}$ and shows the existence of a measure $\mu\in\mathcal{P}_{c}(\T)$
such that \[\inf_{k,k'\ge 0}\mc(2^{k}3^{k'})>0.\]
 In view of this result, one may naturally wonder for which values of $\delta \in(0,1)$ there exists a measure $\mu \in\mathcal{P}_{c}(\T)$ such that 
\begin{equation*}
 \ds{\inf_{\genfrac{}{}{0pt}{}{k,k'\ge 0}{}}\mc(2^{k}3^{k'})\ge\delta},
\end{equation*}

\noindent or, equivalently, whether the Furstenberg set $F$ is a \ka\ set in $\Z$, and if yes, with which (modified) \ka\ constant. 
In this direction, we prove the following result:

\begin{theorem}\label{Th3bis}
Let $F=\{2^k3^{k'}\textrm{ ; }k,k'\ge 0\}$. Then $\Kct(F) \le 1$. More precisely, there exists for every $\delta\in(0,1/2)$ a continuous probability measure $\mu $ on $\T$ with nonnegative Fourier coefficients such that 
\[
\inf_{k,k'\ge 0}\mc(2^{k}3^{k'})>\delta.
\]
\end{theorem}

\subsection*{Rigidity sequences} Our strategy for proving Theorem \ref{Th3} is to construct measures $\mu \in\mathcal{P}_{c}(\T)$ whose Fourier coefficients tend to $1$ along a substantial part of the set $\{p_{1}^{k_{1}}\dots p_{r}^{k_{r}}\,;\,k_{1}\ge 0,\dots,k_{r}\ge 0\}$. In other words, we show that certain large subsets of this set form are, when taken in a strictly increasing order, rigidity sequences in the sense of \cite{BDLR} or \cite{EG}. Recall that a dynamical system $(X,\mathcal{B},m;T)$ on a Borel probability space is called \emph{rigid} if there exists a strictly increasing sequence of integers $(n_{k})_{k\ge 1}$ such that ${||U_{T}^{n_{k}}f-f||}\rightarrow{0}$ as ${k}\rightarrow{+\infty}$ for every $f\in L^{2}(X,\mathcal{B},m)$, where $U_{T}$ denotes as usual the Koopman operator 
${f\mapsto f\circ T}$ associated to $T$ on $L^{2}(X,\mathcal{B},m)$. Equivalently, ${m(T^{-n_{k}}A\bigtriangleup A)}\rightarrow{0}$ as ${k}\rightarrow{+\infty}$ for every $A\in\mathcal{B}$. We say in this case that the system is rigid with respect to the sequence $(n_{k})_{k\ge 1}$, or that $(n_{k})_{k\ge 1}$ is \emph{a rigidity sequence for} $(X,\mathcal{B},m;T)$. The case where the system $(X,\mathcal{B},m;T)$ is weakly mixing is particularly interesting, and is the object of the works 
\cite{BDLR} and \cite{EG}. A strictly increasing sequence $(n_{k})_{k\ge 1}$ of integers is called \emph{a rigidity sequence} if there exists a weakly mixing  system which is rigid with respect to 
$(n_{k})_{k\ge 1}$.
\par\smallskip 
Using Gaussian dynamical systems, one can show that $(n_{k})_{k\ge 1}$ is a rigidity sequence if and only if there exists a measure $\mu \in\mathcal{P}_{c}(\T)$ such that 
${\wh{\mu }(n_{k})}\rightarrow{1}$ as ${k}\rightarrow{+\infty}$. The study of rigidity sequences was initiated in \cite{BDLR} and \cite{EG}. Further works on this topic include  the papers \cite{Aa}, \cite{A}, \cite{AHL}, \cite{G}, \cite{FT} \cite{FK} and \cite{Grie} among others. The paper \cite{FT} by Fayad and Thouvenot is especially relevant here: the authors re-obtain a result of Adams \cite{A}, showing that whenever $(n_{k})_{k\ge 1}$ is a rigidity sequence for an ergodic rotation on the circle, it is a rigidity sequence for a weakly mixing system. The proof of this result in \cite{A} relies on an involved construction of a suitable weakly mixing system by cutting and stacking, while the authors of \cite{FT} proceed by a direct construction of suitable continuous probability measures: they show that if ${\lambda ^{n_{k}}}\rightarrow{1}$ for some $\lambda =e^{2i\pi \theta }\in\T$ with $\theta \in\R\setminus\Q$, there exists $\mu \in\mathcal{P}_{c}(\T)$ such that 
${\wh{\mu }(n_{k})}\rightarrow{1}$.
\par\smallskip
The most important tool for proving Theorems \ref{Th3} and \ref{Th3bis} is the following theorem, which generalizes the result of Fayad and Thouvenot and provides some new examples of non-\ka\ subsets of $\Z$:

\begin{theorem}\label{Th1}
 Let $\nkp{0}$ be a strictly increasing sequence of integers. Suppose that the set
 \[
C=\{\lambda \in\T\ ;{\lambda ^{\nk}\rightarrow 1}\ \textrm{as}\ {k\rightarrow+\infty}\}
\]
is dense in $\T$. Then there exists for every $\varepsilon >0$ a measure $\mu \in\mathcal{P}_{c}(\T)$ such that
 ${\mc(\nk)}\rightarrow{1}$ as
  ${k}\rightarrow{+\infty}$
and $\sup_{k\ge 0}|\mc(\nk)-1|<\varepsilon $.
In particular, $\{n_k\textrm{ ; } k\ge 0\}$ is a non-Kazhdan subset of $\Z$.
\end{theorem}

Notice that $C$, like every subgroup of the circle group, is dense in $\T$ as soon as it is infinite.
We deduce from Theorem \ref{Th1}  the following two-dimensional statement, which is asymmetric and involves a uniformity assumption.

\begin{theorem}\label{Th2}
 Let $(\mk)_{k\ge 0}$ and $(n_{k'})_{k'\ge 0}$ be two strictly increasing sequences of integers.  Let also ${\psi \,:\,\N}\longrightarrow{\N}$ be such that   ${\psi (k)\rightarrow +\infty}$ as
 ${k\rightarrow+\infty}$, and set
\[
D_\psi=\{(k,k')\in\N^{2}\, ;\,0\le k'\le \psi(k)\}.
\]
  Suppose that the set 
  \[
C'_{\psi}=\{\lambda \in\T\ ;{\lambda ^{\mk n_{k'}}\rightarrow1}\ \textrm{as}\ {k\rightarrow+\infty}, \;(k,k')\in D_\psi\}
\]
is dense in $\T$. 
There exists for every $\varepsilon >0$ a measure $\mu \in\mathcal{P}_{c}(\T)$ such that
 ${\mc(\mk n_{k'})}\rightarrow{1}$ as 
  ${k}\rightarrow{+\infty}$ with $(k,k')\in D_\psi$
and
 \[\sup_{k\ge0,\;0\le k'\le \psi(k)}|\mc(\mk n_{k'})-1|<\varepsilon .\]
\par\smallskip
\noindent
 In particular, $\{m_k n_{k'}\textrm{ ; } k\ge 0,\; 0\le k'\le \psi(k)\}$ is a non-Kazhdan subset of $\Z$.
\end{theorem}

Given a doubly indexed sequence $(z_{k,k'})_{k,k'\ge 0}$ of complex numbers, saying that $z_{k,k'}$ converges to $z\in\C$ as ${k}\rightarrow{+\infty}$ with $(k,k')\in D_\psi$ means that
there exists for every $\gamma >0$ an integer $k_{0}$ such that
$|z_{k,k'}-z|<\gamma $ for every $(k,k')\in\N^2$ with $k\ge k_{0}$ and $0\le k'\le \psi(k)$.
\par\smallskip
Remark also that the assumption of Theorem \ref{Th2}
 is in particular satisfied if the set  
 \[
C'=\{\lambda \in\T\ ;{\lambda ^{\mk n_{k'}}\rightarrow 1}\ \textrm{as}\ {k\rightarrow+\infty}\ \textrm{uniformly in}\ k'\}
\]
is dense in $\T$.
\par\smallskip
Theorem \ref{Th1} allows us to retrieve essentially all known examples of rigidity sequences (notable exceptions being the examples of \cite{FK} and \cite{Grie}). We state separately as Corollaries \ref{Cor1} and  \ref{Cor2} the parts of Theorems \ref{Th1} and \ref{Th2} dealing with rigidity sequences:

 \begin{corollary}\label{Cor1}
 Let $\nkp{0}$ be a strictly increasing sequence of integers. Suppose that the set
 \[
C=\{\lambda \in\T\ ;\lambda ^{\nk}\rightarrow 1\ \textrm{as}\ k\rightarrow +\infty\}
\]
is dense in $\T$. 
Then $\nkp{1}$ is a rigidity sequence.
 \end{corollary}
\noindent

 \begin{corollary}\label{Cor2}
  Let $(\mk)_{k\ge 0}$ and $(n_{k'})_{k'\ge 0}$ be two strictly increasing sequences of integers.  Let also ${\psi \,:\,\N}\longrightarrow{\N}$ be such that   ${\psi (k)\rightarrow+\infty}$ as
 ${k\rightarrow+\infty}$, and suppose that the set 
  \[
C'_{\psi}=\{\lambda \in\T\ ;{\lambda ^{\mk n_{k'}}\rightarrow 1}\ \textrm{as}\ {k\rightarrow +\infty}, \;(k,k')\in D_\psi\}
\]
is dense in $\T$. 
Then there exists a continuous probability measure $\mu $ on $\T$ such that ${\mc(\mk n_{k'})\rightarrow 1}$ as ${k\rightarrow+\infty}$ with $(k,k')\in D_\psi$.
 \end{corollary}

The proof of Theorem \ref{Th1} builds on some ideas from \cite{FT}. While being an immediate consequence of Theorem \ref{Th1}, Corollary \ref{Cor1} admits a direct proof which is very much in the spirit of that of the main result of \cite{FT}. As Corollary \ref{Cor1} is of independent interest in the study of rigidity sequences, we will briefly sketch this direct proof in Section \ref{Secajoutee} of the paper.
\par\smallskip
Theorem \ref{Th3} is obtained by first observing that the set $\{p_{1}^{k_{1}}\dots p_{r}^{k_{r}}\,;\,p_1\ge0,\,\dots ,p_r\ge 0\}$
can be split into $r$ sets to which Theorem \ref{Th2} (or Corollary \ref{Cor2}) applies, and then considering a convex combination of the continuous measures obtained in this way.

\subsection*{Organization of the paper} The paper is structured as follows. We give in Section~\ref{Section 3} the proof 
 of Theorems \ref{Th1} and \ref{Th2}, and sketch in Section \ref{Secajoutee} a direct proof of Corollaries \ref{Cor1} and \ref{Cor2}, essentially following the arguments of Fayad and Thouvenot in \cite{FT}.  
In Section \ref{Section 5}, we recall a characterization of generating \ka\ subsets of $\Z$ from \cite{BG1}, and detail the links between several natural constants involved in this characterization. We explain in particular why the generating subsets
of $\Z$ which satisfy the property stated in (C4) are exactly the \ka\ subsets of $\Z$ with modified \ka\ constant $\sqrt{2}$. 
 Section \ref{Section 6} is devoted to applications: we prove Theorems \ref{Th3} and \ref{Th3bis}, and show how to retrieve many examples of rigidity sequences, using Corollaries \ref{Cor1} and \ref{Cor2}. We also provide an application of Theorem \ref{Th3bis} to the study of the size of the exceptional set of values $\theta \in\R$ for which the sequence $(n_{k}\theta )_{k\ge 0}$ is not almost uniformly distributed modulo $1$ with respect to a (finite) complex Borel measure $\nu\in\mathcal{M}(\T)$, where $(n_{k})_{k\ge 0}$ denotes the Furstenberg sequence. Namely, we show that this exceptional set is uncountable, thus providing a new example of a sublacunary sequence with uncountable exceptional set for (almost) uniform distribution.


\section{Proof of Theorems \ref{Th1} and \ref{Th2}}\label{Section 3}

Given two integers $a<b$, we will when the context is clear denote by $[a,b]$ the set of integers $k$ such that $a\le k\le b$.

\begin{proof}[Proof of Theorem \ref{Th1}]
 Fix $\varepsilon \in(0,1/2)$.
The general strategy of the proof is the following: we construct a sequence $(\lambda _{i})_{i\ge 1}$ of pairwise distinct elements of $C$, as well as a strictly increasing sequence of integers $(N_{p})_{p\ge 0}$ and, for every $p\ge 0$, a sequence $(a_i^{(p)})_{1\le i\le 2^p}$ of positive weights with 
$\sum_{i=1}^{2^{p}} a_{i}^{(p)}=1$,
 such that the probability measures
\[
\mu _{p}=\sum_{i=1}^{2^{p}}a_{i}^{(p)}\dpi{\lambda _{i}}
\]
satisfy certain properties stated below. At step $p$, we determine the elements $\lambda _{i}$ for $2^{p-1}<i\le2^{p}$ as well as the integer $N_{p}$ and the weights 
$a_i^{(p)}$, ${1\le i\le 2^p}$, in such a way that $\lambda _{1}=1$ and $a_1^{(0)}=1$, so that $\mu_0= \dpi{1}$, $N_{0}=0$, and
\medskip
\begin{enumerate}[label=(\arabic*)]
\item \label{reference1} for every $p\ge 1$, every $j\in[0,p-1]$ and every $k\in [N_{j},N_{j+1}]$,
\[
\int_{\T}|\lambda ^{\nk}-1|\,d\mu _{p}(\lambda )<3\varepsilon (1-\varepsilon )^{j};
\]
\medskip
\item\label{reference2} for every $p\ge 1$, every $q\in[0,p-1]$, $l\in[1,2^{p-q})$, $r\in[1,2^{q}]$,
\[
|\lambda _{l2^q+r}-\lambda _{r}|<\eta _{q}
\]
where $\eta _{q}=\dfrac{1}{4}\inf_{1\le i<j\le 2^{q}}|\lambda _{i}-\lambda _{j}|$ for every $q\ge 1$, and $\eta_0=1$;
\medskip
\item\label{reference3} for every $p\ge 1$ and every $k\ge N_{p}$,
\[
\int_{\T}|\lambda ^{\nk}-1|\,d\mu _{p}(\lambda )<\varepsilon (1-\varepsilon )^{p+1};
\]
\medskip
\item\label{reference4} for every $p\ge 1$, every $q\in[1,p]$ and every $r\in[1,2^q]$,
\[
\sum_{\{i\in[1,2^{p}]\,;\,i\equiv r\!\!\!\!\mod 2^{q}\}}\mu _{p}(\{\lambda _{i}\})\le(1-\varepsilon )^{q}.
\]
\medskip
\noindent Remark that property \ref{reference2} implies that the sequence $(\lambda _{i})_{i\ge 1}$ satisfies
\smallskip
\item\label{reference5} for every $q\ge 0$, every $l\ge 0$, and every $r\in[1,2^{q}]$,
\[
|\lambda _{l2^{q}+r}-\lambda _{r}|<\eta _{q},
\]
\medskip
\noindent and that property \ref{reference4} applied to $q=p$ yields that
\smallskip
\item\label{reference6} for every $p\ge 1$ and every $i\in[1,2^p]$,
\[ \mu _{p}(\{\lambda _{i}\})\le(1-\varepsilon )^{p}.
\]
\end{enumerate}
\par\medskip
Suppose that the sequences $(\lambda _{i})_{i\ge 1}$, $(N_{p})_{p\ge 0}$ and $(a_i^{(p)})_{1\le i\le 2^p}$, $p\ge 0$, have been constructed so as to satisfy properties \ref{reference1} to \ref{reference4} above, and let $\mu $ be a $w^{*}$-limit point of the sequence $(\mu _{p})_{p\ge 0}$ in $\mathcal{P}(\T)$. Property \ref{reference1} clearly implies that
$\sup_{k\ge 0}|\mc(\nk)-1|\le 3\varepsilon $.

\begin{claim}\label{Claim 1}
We have ${\mc(\nk)}\rightarrow{1}$ as ${k}\rightarrow{+\infty}$.
\end{claim}

\begin{proof}
 For every $k\ge 0$, denote by $j_{k}\ge 0$ the unique integer $j$ such that $k\in[N_{j},N_{j+1})$. For every $p>j_{k}$, we have by \ref{reference1}
 \[
\int_{\T}|\lambda ^{\nk}-1|\,d\mu _{p}(\lambda )<3\varepsilon (1-\varepsilon )^{j_k}\quad \textrm{so that}\quad \int_{\T}|\lambda ^{\nk}-1|\,d\mu (\lambda )\le 3\varepsilon (1-\varepsilon )^{j_k}.
\]
Since ${j_{k}}\rightarrow{+\infty}$ as ${k}\rightarrow{+\infty}$,
${\ds\int_{\T}|\lambda ^{\nk}-1|\,d\mu (\lambda )}\rightarrow {0}$, i.e.\ 
${\mc(\nk)}\rightarrow{1}$.
\end{proof}

\begin{claim}\label{Claim 2}
 The probability measure $\mu $ is continuous.
\end{claim}

\begin{proof}
Fix $q\ge 1$, and consider for every $r\in[1,2^{q}]$ the two arcs $\Gamma _{r}$ and $\Delta_r$ of $\T$ defined by 
\[
\Gamma _{r}=\{\lambda \in\T\,;\,|\lambda -\lambda _{r}|\le\eta _{q}\}\;\;\textrm{ and }\;\;\Delta _{r}=\{\lambda \in\T\,;\,|\lambda -\lambda _{r}|<\frac{3}{2}\eta _{q}\}.
\]
The $2^{q}$ arcs $\Delta_r$ are pairwise disjoint. Indeed, for every $r,r'\in[1,2^{q}]$ with $r\neq r'$, every $\lambda \in\Delta _{r}$ and every $\lambda '\in \Delta _{r'}$, we have by the definition of $\eta _{q}$ that 
\[
|\lambda -\lambda '|\ge|\lambda _{r}-\lambda _{r'}|-3\eta _{q}\ge 4\eta _{q}-3\eta _{q}=\eta _{q}>0.
\]
So $\Delta _{r}$ and $\Delta_{r'}$ do not intersect.
\par\medskip
Let us next estimate the quantity $\mu _{p}(\Gamma _{r})$  for every $r\in[1,2^{q}]$ and every $p\ge q$. We have
\[
\mu _{p}(\Gamma _{r})=\sum_{\{i\in[1,2^{p}]\,;\,\lambda _{i}\in \Gamma_{r}\}}\mu _{p}(\{\lambda _{i}\}).
\]
Every $i\in[1,2^{p}]$ can be written as $i=l2^{q}+s$ for some $l\ge 0$ and $s\in[1,2^{q}]$. By \ref{reference5}, $\lambda _{i}$ belongs to $\Gamma _{s}$. Since the arcs $\Delta _{r'}$, $r'\in[1,2^{q}]$, are pairwise disjoint, it follows that
\[
\mu _{p}(\Delta _{r})=
\mu _{p}(\Gamma _{r})=\sum_{\{i\in[1,2^{p}]\,;\,i\equiv r\!\!\!\!\mod 2^{q}\}}\mu _{p}(\{\lambda _{i}\})\le(1-\varepsilon )^{q}
\]
by \ref{reference4}.
Also,
\[
 \mu_p\Bigl(\bigcup_{r=1}^{2^q}\Gamma_r\Bigr)=1.
\]
Since the arcs $\Gamma_r$ are closed while the arcs $\Delta_r$ are open, going to the limit as $p$ goes to infinity yields that 
$\mu (\Delta _{r})\le (1-\varepsilon)^{q}$ for every $r\in[1,2^{q}]$ and
\[
 \mu\Bigl(\bigcup_{r=1}^{2^q}\Gamma_r\Bigr)=1.
\]
If $\lambda\in\T$ is such that $\mu(\{\lambda\})>0$, there exists  an $r\in[1,2^q]$ such that $\lambda\in \Gamma_r\subset\Delta_r$. So $\mu(\{\lambda\})\le \mu(\Delta_r)\le  (1-\varepsilon)^{q}$, a contradiction if $q$ is sufficiently large.
 It follows that the measure $\mu $ is continuous.
\end{proof}

By Claims \ref{Claim 1} and \ref{Claim 2}, it suffices to construct $(\lambda _{i})_{i\ge 0}$, $(N_{p})_{p\ge 0}$ and $(a_i^{(p)})_{1\le i\le 2^p}$, $p\ge 0$, satisfying properties \ref{reference1} to \ref{reference4} in order to prove Theorem \ref{Th1}. Recall that for $p=0$, we set $\lambda _{1}=1$, $a_1^{(0)}=1$ and $N_{0}=0$, so that $\mu _{0}=\dpi{1}$.
\par\medskip
For $p=1$, we choose $\lambda _{2}\in C$ distinct from $\lambda _{1}$ with $|\lambda _{2}-\lambda _{1}|<1$
and set $\mu _{1}=
(1-\varepsilon )\dpi{1}+\varepsilon \dpi{\lambda _{2}}$. We have for every $k\ge 0$
\[
\int_{\T}|\lambda ^{\nk}-1|\,d\mu _{1}(\lambda )=\varepsilon |\lambda _{2}^{\nk}-1|\le 2\varepsilon <3\varepsilon .
\]
Hence property \ref{reference1} is satisfied whatever the choice of $N_{1}$. Since $\eta_0=1$ and $|\lambda _{2}-\lambda _{1}|<1$, property \ref{reference2}
is satisfied. We now have to choose $N_{1}$
in such a way that property \ref{reference3} is satisfied. Since $\lambda _{2}$ belongs to $C$,
\[
\int_{\T}|\lambda ^{\nk}-1|\,d\mu _{1}(\lambda )=\varepsilon
{|\lambda _{2}^{\nk}-1|}\rightarrow{0}\quad \textrm{as}\quad  {k}\rightarrow{+\infty},
\]
so we can choose $N_{1}$ so large that 
\[
\int_{\T}|\lambda ^{\nk}-1|\,d\mu _{1}(\lambda )<\varepsilon(1-\varepsilon)^{2}\quad \textrm{for every}\ k\ge N_{1}.
\]
Moreover, $\mu _{1}(\{1\})=1-\varepsilon $ and $\mu _{2}
(\{\lambda _{2}\})=\varepsilon <1-\varepsilon $, so \ref{reference4}, which we only need to check for $q=p=1$, is true.
This terminates the construction for $p=1$.
\par\medskip
Suppose now that the construction has been carried out until step $p$, i.e.\ that the quantities $\lambda _{i}$, $i\in[1,2^{p}]$, 
$(a_i^{(l)})_{1\le i\le 2^l}$, and $N_{l}$, $l\in[0,p]$,
 have been constructed satisfying properties \ref{reference1} to \ref{reference4}. 
\par\medskip
We construct by induction on $s\in[1,2^{p}]$ elements $\lambda _{2^{p}+s}$ of $C$, measures $\mu _{p,s}\in\mathcal{P}(\T)$
of the form
\[\mu _{p,s}=
\sum_{i=1}^{2^p+s}b_i^{(p,s)}\,\dpi{\lambda _{i}}\quad \textrm{with} \quad b_i^{(p,s)}>0\quad \textrm{and} \quad \sum_{i=1}^{2^p+s}b_i^{(p,s)}=1,
\]
and integers $N_{p,s}$ in such a way that  the elements $\lambda _{i}$, $i\in[1,2^{p+1}]$, are all distinct, $N_{p}<N_{p,1}<\dots<N_{p,2^{p}}$, and the following five properties are satisfied:
\medskip
\begin{enumerate}[label=(\alph*)]
\item \label{referencea} for every $j\in[0,p-1]$ and every $k\in[N_{j},N_{j+1}]$,
\[
\int_{\T}|\lambda ^{\nk}-1|\,d\mu _{p,s}(\lambda )<3\varepsilon (1-\varepsilon )^{j};
\]
\item \label{referenceb} for every $k\ge N_{p}$,
\[
\int_{\T}|\lambda ^{\nk}-1|\,d\mu _{p,s}(\lambda )<3\varepsilon (1-\varepsilon )^{p};
\]
\item \label{referencec} for every $k\ge N_{p,s}$,
\[
\int_{\T}|\lambda ^{\nk}-1|\,d\mu _{p,s}(\lambda )<3\varepsilon (1-\varepsilon )^{p+2};
\]
\item \label{referenced} $\mu _{p,s}(\{\lambda _{i}\})=\mu _{p}(\{\lambda _{i}\})$ for every $i\in(s,2^{p}]$ and $$\mu _{p,s}(\{\lambda _{i}\})+\mu _{p,s}(\{\lambda _{2^p+i}\})=\mu _{p}(\{\lambda _{i}\}) \quad\textrm{for every } i\in[1,s];$$
\item \label{referencee} $\mu _{p,s}(\{\lambda _{i}\})\le (1-\varepsilon )^{p+1}$
for every $i\in[1,s]\cup [2^{p}+1,2^{p}+s]$.
\end{enumerate}
\par\medskip

\addtocounter{enumi}{1}
  \setcounter{equation}{\value{enumi}}

Let us start with the construction of $\lambda _{2^{p}+1}$. By density of $C$, one can choose $\lambda _{2^{p}+1}$ distinct from all the elements $\lambda _{i}$, $i\in[1,2^{p}]$, with 
$|\lambda _{2^{p}+1}-\lambda _{1}|$ arbitrarily small. We define $\mu _{p,1}$ as 
\begin{align*}
\mu _{p,1}&=\mu _{p}+\mu _{p}(\{1\})\,\varepsilon\, (\dpi{\lambda _{2^{p}+1}}-\dpi{\lambda _{1}}) \\
&= \mu _{p}(\{1\})\,(1-\varepsilon)\,\dpi{\lambda _{1}}+\sum_{i=2}^{2^p}\mu _{p}(\{\lambda_i\})\,\dpi{\lambda _{i}}+\mu _{p}(\{1\})\,\varepsilon\, \dpi{\lambda _{2^{p}+1}}.
\end{align*}
In other words, we split the point mass $\dpi{\lambda _{1}}$ appearing in the expression of $\mu _{p}$ into $(1-\varepsilon)\dpi{\lambda _{1}}+\varepsilon\dpi{\lambda _{2^{p}+1}}$.
We have for every $k\ge 0$
\begin{align}\label{Eq5}
 \int_{\T}|\lambda ^{\nk}-1|\,d\mu _{p,1}(\lambda )&\le\int_{\T}|\lambda ^{\nk}-1|\,d\mu _{p}(\lambda )+\mu _{p}(\{1\})\,\varepsilon\, |\lambda _{2^{p}+1}^{\nk}-\lambda _{1}^{\nk}|\\
 &\le\int_{\T}|\lambda ^{\nk}-1|\,d\mu _{p}(\lambda )+(1-\varepsilon )^{p}\,\varepsilon\, |\lambda _{2^{p}+1}^{\nk}-\lambda _{1}^{\nk}|
\notag
\end{align}
since 
$\mu _{p}(\{1\})\le (1-\varepsilon )^{p}$ by \ref{reference6}.
If $|\lambda _{2^{p}+1}-\lambda _{1}|$ is sufficiently small, we have by \ref{reference1} that 
\[
\int_{\T}|\lambda ^{\nk}-1|\,d\mu _{p,1}(\lambda )<3\varepsilon (1-\varepsilon )^{j}
\]
for every $j\in[0,p-1]$ and every $k\in[N_{j},N_{j+1}]$ (the set of pairs of integers $(j,k)$ with 
$j\in [0,p-1]$ and $k\in[N_{j},N_{j+1}]$ is finite).
So \ref{referencea} holds true. Also, (\ref{Eq5}) and \ref{reference3} imply that for every $k\ge N_{p}$,
\[
\int_{\T}|\lambda ^{\nk}-1|\,d\mu _{p,1}(\lambda )<\varepsilon \,(1-\varepsilon )^{p+1}+2\varepsilon \,(1-\varepsilon )^{p}<3\varepsilon \,(1-\varepsilon )^{p}
\]
so that \ref{referenceb} holds true.
Since all the elements $\lambda _{i}$, $i\in[1,2^{p}+1]$, belong to $C$, there exists $N_{p,1}>N_{p}$ such that 
\[
\int_{\T}|\lambda ^{\nk}-1|\,d\mu _{p,1}(\lambda )<3\varepsilon (1-\varepsilon )^{p+2}\quad \textrm{for every}\ k\ge N_{p,1}.
\]
Property \ref{referenced} is clear from the expression of $\mu _{p,1}$, and property \ref{referencee} is satisfied since $\mu _{p,1}(\{1\})=\mu _{p}(\{1\})\,(1-\varepsilon)\le (1-\varepsilon)^{p+1}$ and
$\mu _{p,1}(\{\lambda _{2^{p}+1}\})=\mu _{p}(\{1\})\,\varepsilon
\le \varepsilon\, (1-\varepsilon)^{p}\le (1-\varepsilon)^{p+1}$ by \ref{reference6}.
Properties \ref{referencea} to \ref{referencee} are thus satisfied for $s=1$.
\par\smallskip
Suppose now that 
$\lambda _{2^{p}+s'}$, $\mu _{2^{p}+s'}$, and $N_{2^{p}+s'}$ have been constructed for $s'<s$. Let $\lambda _{2^{p}+s}\in C\setminus\{\lambda _{1},\dots,\lambda _{2^{p}+s-1}\}$ be very close to $\lambda _{s}$, and set
\begin{equation}\label{qqch}
 \mu _{p,s}=\mu _{p,s-1}+\mu _{p,s-1}(\{\lambda _{s}\})\,\varepsilon \,(\dpi{\lambda _{2^{p}+s}}-\dpi{\lambda _{s}}).
\end{equation}
This time, the  point mass $\dpi{\lambda _{s}}$ appearing in $\mu _{p}$ is split as
$(1-\varepsilon)\dpi{\lambda _{s}}+\varepsilon\dpi{\lambda _{2^{p}+s}}$. Since, by \ref{reference6}, 
\begin{equation}\label{Eq6}
  \int_{\T}|\lambda ^{\nk}-1|\,d\mu _{p,s}(\lambda )\le \int_{\T}|\lambda ^{\nk}-1|\,d\mu _{p,s-1}(\lambda )+(1-\varepsilon )^{p}\,\varepsilon \,|\lambda ^{\nk}_{2^{p}+s}-\lambda ^{\nk}_{s}|,
\end{equation}
for every $k\ge 0$,
the induction assumption implies that \ref{referencea} holds true provided $|\lambda _{2^{p}+s}-\lambda _{s}|$ is sufficiently small. As to \ref{referenceb}, we have to consider separately the cases $N_{p}\le k< N_{p,s-1}$ and $k\ge N_{p,s-1}$. If $|\lambda _{2^{p}+s}-\lambda _{s}|$ is sufficiently small, we have by (\ref{Eq6}) and \ref{referenceb} for $s-1$ that
\[
\int_{\T}|\lambda ^{\nk}-1|\,d\mu _{p,s}(\lambda )<3\varepsilon \,(1-\varepsilon )^{p}\quad \textrm{for every}\ N_{p}\le k< N_{p,s-1}.
\]
By property \ref{referencec} at step $s-1$ and (\ref{Eq6}),
\[
\int_{\T}|\lambda ^{\nk}-1|\,d\mu _{p,s}(\lambda )<\varepsilon \,(1-\varepsilon )^{p+2}+2\varepsilon \,(1-\varepsilon )^{p}<3\varepsilon \,(1-\varepsilon )^{p}
\]
for every $k\ge N_{p,s-1}$. Hence \ref{referenceb} is satisfied at step $s$. Property \ref{referencec} is satisfied if $N_{p,s}$ is chosen sufficiently large since all the elements $\lambda _{i}$, $i\in[1,2^{p}+s]$, belong to $C$. 
\par\smallskip
Property \ref{referenced} follows from (\ref{qqch}) and property \ref{referenced} at step $s-1$.
Indeed, $\mu _{p,s}(\{\lambda _{i}\})=\mu _{p,s-1}(\{\lambda _{i}\})$ for every $i\not \in\{s,2^{p}+s \}$. Also,
$\mu _{p,s-1}(\{\lambda _{i}\})=\mu _{p}(\{\lambda _{i}\})$ for every $i\in [s,2^p]$, so that $\mu _{p,s}(\{\lambda _{i}\})=\mu _{p}(\{\lambda _{i}\})$ for every $i\in (s,2^p]$. Observe next that $\mu _{p,s}(\{\lambda _{i}\})+\mu _{p,s}(\{\lambda _{2^p+i}\})=\mu _{p,s-1}(\{\lambda _{i}\})+\mu _{p,s-1}(\{\lambda _{2^p+i}\})=
\mu _{p}(\{\lambda _{i}\})$ for every $i\in[1,s-1]$. Lastly,
$\mu _{p,s}(\{\lambda _{s}\})+\mu _{p,s}(\{\lambda _{2^p+s}\})=\mu _{p,s-1}(\{\lambda _{s}\})=\mu _{p}(\{\lambda _{s}\})$. So property \ref{referenced} is true at step $s$.
\par\smallskip
As to property \ref{referencee}, we have 
$\mu _{p,s}(\{\lambda _{i}\})=\mu _{p,s-1}(\{\lambda _{i}\})$ for every $i\not \in\{s,2^{p}+s \}$. So $\mu _{p,s}(\{\lambda _{i}\})\le(1-\varepsilon )^{p+1}$ for every $i\in[1,s)\cup [2^{p}+1,2^{p}+s)$. Also
\[
\mu _{p,s}(\{\lambda _{s}\})=\mu _{p,s-1}(\{\lambda _{s}\})\,(1-\varepsilon )=\mu _{p}(\{\lambda _{s}\})\,(1-\varepsilon )\le (1-\varepsilon )^{p+1}
\]
by \ref{reference6},
while $\mu _{p,s}(\{\lambda _{2^{p}+s}\})=\mu _{p,s-1}(\{\lambda _{s}\})\,\varepsilon \le(1-\varepsilon )^{p+1}$, again by \ref{reference6}. So \ref{referencee} holds true at step $s$. This terminates the construction of the measures $\mu _{p,s}$.
\par\medskip
Let us now set $\mu _{p+1}=\mu _{p,2^{p}}$ and $N_{p+1}=N_{p,2^{p}}$. It remains to check that with these choices of $\lambda _{i}$, $i\in[1,2^{p+1}]$, $\mu _{p+1}$ and $N_{p+1}$, properties \ref{reference1} to \ref{reference4} are satisfied.
\par\medskip
By \ref{referencea}, property \ref{reference1} is satisfied for every $j\in[0,p-1]$. The case where $j=p$ follows from \ref{referenceb}. So \ref{reference1} is true. Property \ref{reference3} follows immediately from \ref{referencec}.
Property \ref{reference4} is a consequence of \ref{referenced} and \ref{referencee}. Indeed, suppose first that $q\in[1,p]$. Then
\begin{align*}
 \sum_{\{i\in[1,2^{p+1}]\,;\,i\equiv r\!\!\!\!\mod 2^{q}\}}\mu _{p+1}(\{\lambda _{i}\})
&=\sum_{\{i\in[1,2^{p}]\,;\,i\equiv r\!\!\!\!\mod 2^{q}\}}\left(\mu _{p+1}(\{\lambda _{i}\})+\mu _{p+1}(\{\lambda _{2^p+i}\})\right)\\
&=\sum_{\{i\in[1,2^{p}]\,;\,i\equiv r\!\!\!\!\mod 2^{q}\}}\mu _{p}(\{\lambda _{i}\})
\le(1-\varepsilon )^{q}
\end{align*}
by \ref{referenced} above and \ref{reference4} at step $p$. If $q=p+1$, \ref{reference4} follows immediately from \ref{referencee}.
So it only remains to check \ref{reference2}.
\par\medskip
Fix $q\in[0,p]$, $l\in[1,2^{p+1-q})$ and $r\in[1,2^{q}]$. Consider first the case where $q=p$. In this case $l=1$, and the quantities under consideration have the form $|\lambda _{2^{p}+r}-\lambda _{r}|$, with $r\in[1,2^{p}]$. One can ensure in the construction that $|\lambda _{2^{p}+r}-\lambda _{r}|<\eta _{p}$ for every $r\in[1,2^{p}]$ and then \ref{reference2} holds true for $q=p$.
\par\medskip
Suppose then that $q\in[0,p-1]$, and write $l$ as $l=l'+\varepsilon 2^{p-q}$ with $\varepsilon \in\{0,1\}$ and $l'\in[1,2^{p-q})$. Then $l2^{q}+r=l'2^{q}+r+\varepsilon 2^{p} $. Set $s=l'2^{q}+r$. Then $1\le s\le (2^{p-q}-1)2^{q}+2^{q}=2^{p}$, i.e. 
$s\in[1,2^{p}]$. We have
\[
|\lambda _{l2^{q}+r}-\lambda _{r}|\le|\lambda _{s+\varepsilon 2^{p}}-\lambda _{s}|+|\lambda _{l'2^{q}+r}-\lambda _{r}|.
\]
If $\varepsilon =0$, the first term is zero; if $\varepsilon =1$, it is equal to 
$|\lambda _{2^{p}+s}-\lambda _{s}|$, which can be assumed to be as small as we wish in the construction. As to the second term, it is less  than $\eta _{q}$ by property \ref{reference2} at step $p$, since $l'\in[1,2^{p-q})$ and $r\in[1,2^{q}]$ with $q\in[0,{p-1}]$. We can thus ensure that
\[
|\lambda _{l2^{q}+r}-\lambda _{r}|<\eta _{q}
\]
for every $q\in[0,p]$, $l\in[1,2^{p+1-q})$, and $r\in[1,2^{q}]$. So property \ref{reference2} is satisfied at step $p+1$, and this concludes the proof of Theorem \ref{Th1}.
\end{proof}

Theorem \ref{Th2} is now a formal consequence of Theorem \ref{Th1}.

\begin{proof}[Proof of Theorem \ref{Th2}] Recall that $
D_\psi=\{(k,k')\in\N^{2}\, ;\,0\le k'\le \psi(k)\}$
  and
  \[
C'_{\psi}=\{\lambda \in\T\ ;{\lambda ^{\mk n_{k'}}\rightarrow 1}\ \textrm{as}\ {k\rightarrow +\infty}, \;(k,k')\in D_\psi\}.
\]
 Order the set $\{m_k n_{k'}\textrm{ ; } (k,k')\in D_{\psi}\}$ as a strictly increasing sequence $(p_l)_{l\ge 0}$ of integers. Since there exists for every integer $k_1\ge 0$ an integer $l_1\ge 0$ such that $$\{p_l\textrm{ ; }l\ge l_1\}\subseteq\{m_k n_{k'}\textrm{ ; } (k,k')\in D_{\psi},\; k\ge k_1\},$$ every element $\lambda\in C'_\psi$ has the property that 
${\lambda ^{p_l}\rightarrow 1}\ \textrm{as}\ {l\rightarrow +\infty}$. By Theorem \ref{Th1} applied to the sequence $(p_l)_{l\ge 1}$, there exists for every $\varepsilon>0$ a measure $\mu\in\mathcal{P}_c(\T)$ such that ${\mc(p_l)\rightarrow 1}$
as ${l\rightarrow +\infty}$ and
$\sup_{l\ge 0}|\mc(p_l)-1|<\varepsilon $. Then
\[\sup_{k\ge0,\;0\le k'\le \psi(k)}|\mc(\mk n_{k'})-1|<\varepsilon .\]
 Using this time the fact that there exists for every integer $l_2\ge 0$ an integer $k_2\ge 0$ such that $$\{m_k n_{k'}\textrm{ ; } (k,k')\in D_{\psi},\; k\ge k_2\}\subseteq\{p_l\textrm{ ; }l\ge l_2\},$$ we deduce that ${\mc(\mk n_{k'})\rightarrow 1}$ as ${k\rightarrow +\infty}$ with $(k,k')\in D_{\psi}$. Theorem \ref{Th2} is proved.
\end{proof}

\section{A direct proof of Corollaries \ref{Cor1} and \ref{Cor2}}\label{Secajoutee}
We sketch in this section a direct proof of Corollary \ref{Cor1} (Corollary \ref{Cor2} is a formal consequence of it), following almost step by step the construction given in \cite{FT} and bypassing the additional technical difficulties of the proof of Theorem \ref{Th1}.

\begin{proof} Using the notation of the proof of Theorem \ref{Th1}, 
we construct a sequence $(\lambda _{i})_{i\ge 1}$ of pairwise distinct elements of $C$, as well as a strictly increasing sequence of integers $(N_{p})_{p\ge 0}$, such that the measures
\[
\mu _{p}=2^{-p}\sum_{i=1}^{2^{p}}\dpi{\lambda _{i}},\quad p\ge 0
\]
satisfy 
\medskip
\begin{enumerate}[label=(\arabic*')]
\item \label{reference1'} for every $p\ge 1$, every $j\in[0,p-1]$ and every $k\in[N_{j},N_{j+1}]$,
\[
\int_{\T}|\lambda ^{\nk}-1|\,d\mu _{p}(\lambda )<2^{-(j-1)};
\]
\medskip
\item \label{reference2'} for every $p\ge 1$, every $q\in[0,p-1]$, $l\in[1,2^{p-q})$, $r\in[1,2^{q}]$,
\[
|\lambda _{l2^q+r}-\lambda _{r}|<\eta _{q}
\]
where $\eta _{q}=\dfrac{1}{4}\inf_{1\le i<j\le 2^{q}}|\lambda _{i}-\lambda _{j}|$ for every $q\ge 1$, and $\eta_0=1$;
\medskip
\item \label{reference3'} for every $p\ge 1$ and every $k\ge N_{p}$,
\[
\int_{\T}|\lambda ^{\nk}-1|\,d\mu _{p}(\lambda )<2^{-(p+1)}.
\]
\medskip
\noindent Again, property \ref{reference2'} implies that
\smallskip
\item \label{reference4'} for every $q\ge 0$, every $l\ge 0$, and every $r\in[1,2^{q}]$,
\[
|\lambda _{l2^{q}+r}-\lambda _{r}|<\eta _{q}.
\]
\end{enumerate}
\par\medskip
Then an argument similar to the one given in the proof of Theorem \ref{Th1} shows that any $w^{*}$-limit point $\mu $ of $(\mu _{p})_{p\ge 0}$ will be a continuous measure which satisfies
${\mc(\nk)}\rightarrow{1}$ as ${k}\rightarrow{+\infty}$. 
\par\smallskip
For $p=0$, we set $\lambda _{1}=1$, $N_{0}=0$, and $\mu _{0}=\dpi{1}$.
For $p=1$, we choose $\lambda _{2}\in C\setminus\{\lambda _{1}\}$ with $|\lambda _{2}-\lambda _{1}|<1$ and set
$\mu _{1}=\frac{1}{2}(\dpi{1}+\dpi{\lambda _{2}})$. We have
\[
\int_{\T}|\lambda ^{\nk}-1|\,d\mu _{1}(\lambda )=\dfrac{1}{2}\,|\lambda _{2}^{\nk}-1|\le 1<2\quad \textrm{for every}\ k\ge 0.
\]
Hence property \ref{reference1'} is satisfied whatever the choice of $N_{1}$.
Since $|\lambda _{2}-\lambda _{1}|<1$, \ref{reference2'} is true.
If $N_{1}$ is chosen sufficiently large, $\mu _{1}$ satisfies \ref{reference3'}.
\par\medskip
Suppose now that the construction has been carried out until step $p$.
We can then construct by induction on $s\in[1,2^{p}] $ measures $\mu _{p,s}$ which satisfy
\medskip 
\begin{enumerate}[label=(\alph*')]
\item \label{referencea'} every $j\in[0,p-1]$ and every $k\in[N_{j},N_{j+1}]$,
\[
\int_{\T}|\lambda ^{\nk}-1|\,d\mu _{p,s}(\lambda )<2^{-(j-1)};
\]
\item \label{referenceb'}for every $k\ge N_{p}$,
\[
\int_{\T}|\lambda ^{\nk}-1|\,d\mu _{p,s}(\lambda )<2^{-(p-1)};
\]
\item \label{referencec'} for every $k\ge N_{p,s}$,
\[
\int_{\T}|\lambda ^{\nk}-1|\,d\mu _{p,s}(\lambda )<2^{-(p+2)}.
\]
\end{enumerate}
\par\medskip
We define $\mu _{p,1}$ as 
\begin{align*}
\mu _{p,1}&=\mu _{p}+2^{-(p+1)}\bigl (\dpi{\lambda _{2^{p}+1}}-\dpi{\lambda _{1}} \bigr)
\end{align*}
 where $\lambda _{2^{p}+1}\in C\setminus\{\lambda _{1},\dots,\lambda _{2^{p}}\}$ is such that $|\lambda _{2^{p}+1}-\lambda _{1}|$ is very small. Then for every $k\ge 0$,
\begin{equation}\label{Eq9}
  \int_{\T}|\lambda ^{\nk}-1|\,d\mu _{p,1}(\lambda )\le\int_{\T}|\lambda^{\nk}-1|\,d\mu _{p}(\lambda )+2^{-(p+1)}\,|\lambda ^{\nk}_{2^{p}+1}-\lambda _{1}^{\nk}|.
\end{equation}
It follows that \ref{referencea'} holds true for $\mu _{p,1}$, provided that 
$|\lambda _{2^{p}+1}-\lambda _{1}|$ is sufficiently small. Also, we have by (\ref{Eq9}) and \ref{reference3'} that for every $k\ge N_{p}$,
\[
\int_{\T}|\lambda ^{\nk}-1|\,d\mu _{p,1}(\lambda )<2^{-(p+1)}+2^{-p}<2^{-(p-1)}
\]
which is \ref{referenceb'}. If $N_{p,1}$ is sufficiently large, \ref{referencec'} is true.
\par\medskip
Supposing now that $s\ge 2$ and that the construction has been carried out for every $s'<s$, we choose $\lambda _{2^{p}+s}\in C\setminus\{\lambda _{1},\dots,\lambda _{2^{p}+s-1}\}$ very close to $\lambda _{s}$, and set
\[
\mu _{p,s}=\mu _{p,s-1}+2^{-(p+1)}\bigl (\dpi{\lambda _{2^{p}+s}}-\dpi{\lambda _{s}} \bigr).
\]
Since, for every $k\ge 0$,
\begin{equation}\label{Eq10}
  \int_{\T}|\lambda ^{\nk}-1|\,d\mu _{p,s}(\lambda )\le \int_{\T}|\lambda ^{\nk}-1|\,d\mu _{p,s-1}(\lambda )+2^{-(p+1)}|\lambda ^{\nk}_{2^{p}+s}-\lambda ^{\nk}_{s}|,
\end{equation}
the induction assumption implies that \ref{referencea'} holds true provided $|\lambda _{2^{p}+s}-\lambda _{s}|$ is sufficiently small. As to \ref{referenceb'}, we consider separately the cases $N_{p}\le k< N_{p,s-1}$ and $k\ge N_{p,s-1}$. If $|\lambda _{2^{p}+s}-\lambda _{s}|$ is sufficiently small,
\[
\int_{\T}|\lambda ^{\nk}-1|\,d\mu _{p,s}(\lambda )<2^{-(p-1)}\quad \textrm{for every}\ N_{p}\le k<N_{p,s-1}.
\]
By property \ref{referencec'} at step $s-1$ and (\ref{Eq10}),
\[
\int_{\T}|\lambda ^{\nk}-1|\,d\mu _{p,s}(\lambda )<2^{-(p+2)}+2^{-p}<2^{-(p-1)}
\]
for every $k\ge N_{p,s-1}$. Hence \ref{referenceb'} is satisfied at step $s$. Property \ref{referencec'} is satisfied if $N_{p,s}$ is chosen sufficiently large.
This terminates the construction of the measures $\mu _{p,s}$.
\par\smallskip
We then set $\mu _{p+1}=\mu _{p,2^{p}}$ and $N_{p+1}=N_{p,2^{p}}$ and check as in the proof of Theorem \ref{Th1} that properties \ref{reference1'}, \ref{reference2'}, and \ref{reference3'} are satisfied.
\end{proof}

  \par\smallskip

\begin{remark}\label{Remark 6}
 Suppose that the set 
  \[
C'=\{\lambda \in\T\ ;{\lambda ^{\mk n_{k'}}\rightarrow 1}\ \textrm{as}\ {k \rightarrow +\infty}\ \textrm{uniformly in}\ k'\}
\]
is dense in $\T$. 
 It is natural to wonder whether there exists a measure $\mu \in\mathcal{P}_{c}(\T)$ such that ${\mc(\mk n_{k'})}\rightarrow{1}$ as 
 ${k}\rightarrow{+\infty}$ uniformly in $k'$. The following example shows that it is not the case: set $\mk=2^{k}$ and $n_{k'}=k'$ for every $k,k'\ge 0$. The set
 \[
C'=\{\lambda \in\T\,;\,{\lambda ^{\mk n_{k'}}}\rightarrow{1}\ \textrm{as}\ {k}\rightarrow{+\infty}\ \textrm{uniformly in}\ k'\}
\]
contains all $2^{k}$-th roots of $1$, and so is dense in $\T$. Suppose that $\mu \in\mathcal{P}(\T)$ is such that ${\mc(2^{k}k')}\rightarrow{1}$ as ${k}\rightarrow{+\infty}$ uniformly in $k'$. Then there exists an integer $k_{0}\ge 1$ such that $|\mc(2^{k_{0}}k')|\ge 1/2$ for every $k'\ge 0$. Consider 
the measure $\nu =T_{2^{k_{0}}}(\mu) $. Since $\wh{\nu }(n)=\mc(2^{k_{0}}n)$ for every $n\in\Z$, $\nu$ cannot be continuous. Also, $\nu (\{\lambda _0\})=\mu (\{\lambda\in\T\textrm{ ; } \lambda^{2^{k_0}}=\lambda_0\})$ for every $\lambda_0 \in\T$, and so the measure $\mu $ itself cannot be continuous.
\par\smallskip
So the conclusion of Corollary \ref{Cor2} seems to be essentially optimal.
\end{remark}

\section{From Conjecture (C4) to the study of some non-Kazhdan subsets of $\mathbb{Z}$}\label{Section 5}
\subsection{\ka\ constants and Fourier coefficients of probability measures}
We begin this section by recalling a characterization of generating \ka\ subsets of $\Z$, obtained in \cite[Th.\,6.1]{BG1} (see also \cite[Th.\,4.12]{BG2}) and presenting some facts concerning the (modified) \ka\ constants of such sets. We state it here in a slightly modified way (condition (ii) is not exactly the same as in \cite[Th.\,4.12]{BG2}), and include a discussion concerning the links between the various constants appearing in the equivalent conditions. 

\begin{theorem}\label{Th B}
Let $Q$ be a generating subset of $\Z$. Then $Q$ is a \ka\ subset of $\Z$ if and only if one of the following equivalent assertions holds true:
\begin{enumerate}
\item[\emph{(i)}] there exists $\varepsilon \in(0,\sqrt{2})$ such that $(Q,\varepsilon )$ is a modified \ka\ pair. Equivalently, $\Kct(Q)\ge \varepsilon$;
\item[\emph{(ii)}] there exists $\gamma \in(0,1)$ such that any measure $\mu \in\mathcal{P}(\T)$ with $\sup_{n\in Q}(1-\Re e\,\wh{\mu }(n))<\gamma $ has a discrete part;
\item[\emph{(iii)}] there exists $\delta \in(0,1)$ such that any measure $\mu \in\mathcal{P}(\T)$ with $\inf_{n\in Q}|\wh{\mu }(n)|>\delta $ has a discrete part.
\end{enumerate}
Moreover:
\par\smallskip
-- \emph{(i)} is satisfied for $\varepsilon \in(0,\sqrt{2})$ if and only if \emph{(ii)} is satisfied for $\gamma =\varepsilon ^{2}/2$;
\par\smallskip
-- if \emph{(ii)} is satisfied for $\gamma\in(0,1)$, \emph{(iii)} is satisfied for $\delta=\sqrt{1-\gamma}$, while if \emph{(iii)} is satisfied for $\delta\in(0,1)$, \emph{(ii)} is satisfied for $\gamma={1-\delta}$;
\par\smallskip
-- hence if \emph{(i)} is satisfied for $\varepsilon \in(0,\sqrt{2})$,
 \emph{(iii)} is satisfied for $\delta =\sqrt{\strut 1-\varepsilon ^{2}/2}$ , while if \emph{(iii)} is satisfied for $\delta\in(0,1)$, \emph{(i)} holds true for $\varepsilon =\sqrt{2(1-\delta)}$.
\end{theorem}

We prove briefly here the statement concerning the relations between the constants $\varepsilon $, $\gamma $, and $\delta $ appearing in (i), (ii), and (iii) respectively, following \cite{BG1} and \cite{BG2}.

\begin{proof}
 Suppose that (i) is satisfied for $\varepsilon \in(0,\sqrt{2})$, and let $\mu \in\mathcal{P}(\T)$. Consider the unitary operator $U=M_{\lambda }$ of multiplication by $\lambda $ on $L^{2}(\T,\mu )$. Let $f$ be the function constantly equal to $1$. Then 
 $||U^{n}f-f||^{2}=2(1-\Re e\,\wh{\mu }(n))$. If $\sup_{n\in Q}(1-\Re e\,\wh{\mu }(n))<\varepsilon ^{2}/2$, $U$ has an eigenvalue since $\Kct(Q)\ge \varepsilon$, and so $\mu$ has a discrete part.
 \par\smallskip
Conversely, suppose that (ii) is satisfied  for $\gamma \in(0,1)$. Let $U$ be a unitary  operator on a separable Hilbert space $H$, and let $x\in H$ with $||x||=1$ be such that
 \[\sup_{n\in Q}||U^{n}x-x||<\sqrt{ 2\gamma}.\] The proof of \cite[Th.\,4.6]{BG2} shows then that there exists $\mu \in\mathcal{P}(\T)$ such that
 \[
2\sup_{n\in Q}(1-\Re e\,\wh{\mu }(n))=\sup_{n\in Q}||U^{n}x-x||^{2}<2\gamma.
\]
So $\sup_{n\in Q}(1-\Re e\,\wh{\mu }(n))<\gamma $. By (ii), $\mu $ has a discrete part, and so $U$ has an eigenvalue. Hence $\Kct(Q)\ge \sqrt{ 2\gamma}.$
\par\smallskip
 Suppose next that property (ii) is satisfied for $\gamma \in(0,1)$. Let $\mu \in\mathcal{P}(\T )$ be such that $\inf_{n\in Q}|\wh{\mu }(n)|>\sqrt{1-\gamma} $. Set $\nu =\mu *\widetilde{\mu }$. Then $\inf_{n\in Q}\wh{\nu }(n)>1-\gamma $. It follows that
 $\sup_{n\in Q}(1-\wh{\nu }(n))<\gamma $, and $\nu $ has a discrete part. So $\mu $ itself has a discrete part.
 \par\smallskip
 Lastly, suppose that  (iii) is satisfied for $\delta \in(0,1)$. Let $\mu \in\mathcal{P}(\T )$ be a measure satisfying $\sup_{n\in Q}(1-\Re e\,\wh{\mu }(n))<{1-\delta} $. Then $\inf_{n\in Q}|\wh{\mu }(n)|\ge \inf_{n\in Q}\Re e\,\wh{\mu }(n) >\delta $, so $\mu $ has a discrete part.
\end{proof}

\begin{remark}
 Given a subset $Q$ of $\Z$, one can prove, using the spectral theorem for unitary operators, that the following assertions are equivalent (see \cite[Th.\,4.6]{BG2}):
 \par\smallskip
 \begin{enumerate}
\item[{(i')}] $Q$ is a \ka\ subset of $\Z$, i.e.\ there exists $\varepsilon \in(0,\sqrt{2})$ such that $(Q,\varepsilon )$ is a \ka\ pair;
\item[{(ii')}] there exists $\gamma \in(0,1)$ such that any measure $\mu \in\mathcal{P}(\T)$ with $\sup_{n\in Q}(1-\Re e\,\wh{\mu }(n))<\gamma $ is such that $\mu (\{1\})>0$.
\end{enumerate}
Moreover {(i')} holds true for a certain constant $\varepsilon \in (0,\sqrt{2})$ (i.e. $\Kc(Q)\ge \varepsilon$) if and only if  {(ii')} holds true for $\gamma =\varepsilon ^{2}/2$.
\par\smallskip
It is interesting to note that these two conditions (i') and (ii') are not equivalent to the natural version (iii') of (iii) (namely, that there exists $\delta \in(0,1)$ such that any measure $\mu \in\mathcal{P}(\T)$ with $\inf_{n\in Q}|\wh{\mu }(n)|>\delta $ satisfies $\mu (\{1\})>0$). Indeed, (iii') is satisfied for any Dirac mass
$\delta_{\{\lambda\}}$, $\lambda\in\T$. The proof that (ii) implies (iii) in Theorem \ref{Th B} above uses in a crucial way the fact that if 
$\mu\in\mathcal{P}(\T)$ is such that $\mu *\widetilde{\mu }$ has a discrete part, $\mu$ itself has a discrete part. But $\mu *\widetilde{\mu }$ may very well satisfy $\mu *\widetilde{\mu }(\{1\})>0$ while $\mu(\{1\})=0$, and so (ii') does not imply (iii').
\end{remark}

Theorem \ref{Th B} is related to Conjecture (C4) in the following way:

\begin{corollary}\label{Cor C}
 Let $Q$ be a generating subset of $\Z$. The following assertions are equi\-va\-lent:
 \begin{enumerate}
\item[\emph{($\alpha $)}] $Q$ is a \ka\ subset of $\Z$ with $\Kct(Q) =\sqrt{2}$;
\item[\emph{($\beta $)}] any measure $\mu \in\mathcal{P}_{c}(\T)$ satisfies $\inf_{n\in Q}|\wh{\mu }(n)|=0$;
\item[\emph{($\gamma $)}] any measure $\mu \in\mathcal{P}_{c}(\T)$ satisfies 
$\liminf_{\genfrac{}{}{0pt}{}{|n|\to+\infty}{n\in Q}}|\wh{\mu }(n)|=0$.
\end{enumerate}
\end{corollary}

\begin{proof} The equivalence between ($\alpha $) and ($\beta $) follows immediately from Theorem \ref{Th B}.
 So only the implication ($\beta $)$\Longrightarrow$($\gamma $) requires a proof. Suppose that any $\mu \in\mathcal{P}_c(\T)$ satis\-fies $\inf_{n\in Q}|\wh{\mu }(n)|=0$. We want to show that  the conclusion can be reinforced into $\liminf_{\genfrac{}{}{0pt}{}{|n|\to+\infty}{n\in Q}}|\wh{\mu }(n)|=0$. Let $\rho \in\mathcal{P}_c(\T)$ be a Rajchman measure with positive coefficients, that is such that $\lim_{|n|\to+\infty}\wh{\rho }(n)=0$ and $\wh{\rho }(n)>0$ for every $n\in\Z$. Consider the measure $\nu =(\mu *\ti{\mu }+\rho )/2$. It is continuous and satisfies $\nu (n)>0$ for every 
$ n\in\Z$. Since $\inf_{n\in Q}\wh{\nu}(n)=0$ and $\nu (n)>0$ for every 
$ n\in\Z$, $\liminf_{\genfrac{}{}{0pt}{}{|n|\to+\infty}{n\in Q}}\wh{\nu }(n)=0$. Hence $\liminf_{\genfrac{}{}{0pt}{}{|n|\to+\infty}{n\in Q}}|\wh{\mu }(n)|^{2}=0$, and the conclusion follows. 
\end{proof}

So Conjecture (C4) is equivalent to the statement that any non-lacunary  semigroup of integers has modified \ka\ constant $\sqrt{2}$. We can also estimate the Fourier coefficients of a continuous probability measure on $\T$ which is $T_2$- and $T_3$-invariant in terms of the modified Kazhdan constant $\tilde{\kappa}>0$ of the Furstenberg set. Notice that Proposition \ref{prop:coeff} is meaningful only if $\tilde{\kappa}>0$.

\begin{proposition}\label{prop:coeff}
Let $F=\{2^k3^{k'}\textrm{ ; }k,k'\ge 0\}$ and set $\tilde{\kappa} = \Kct(F)$. Let $\mu$ be a continuous probability measure on $\T$ which is $T_2$- and $T_3$-invariant. Then
$$ \left|\hat{\mu}(j)\right| \le  1-\frac{\tilde{\kappa}^2}{2} \quad \textrm{ for every }j\in\Z\setminus\{0\}.$$
\end{proposition}

\begin{proof}
Set, for every $j\in\Z\setminus\{0\}$, $\mu_j=T_j\mu$. Then $\mu_j$ is a continuous measure which satisfies $\hat{\mu}_j(2^k3^{k'})= \hat{\mu}(j)$ for every $k,k'\ge 0$ It follows that if $\delta\in (0,1)$ is such that (iii) of Theorem \ref{Th B} is satisfied, $\delta\ge \left|\hat{\mu}(j)\right|$. Hence, by Theorem \ref{Th B} again, $\tilde{\kappa} \le \sqrt{2(1-\left|\hat{\mu}(j)\right|)}$.
\end{proof}

\begin{remark}\label{Remark1bis}
Although a generating subset $Q$ of $\Z$ is a \ka\ set if and only if $\Kct(Q)>0$, there is no link between the \ka\ constant and the modified \ka\ constant of $Q$. Indeed, there exist \ka\ subsets $Q$ of $\Z$ with maximal modified constant $\Kct(Q) =\sqrt{2}$ and arbitrarily small \ka\ constant $\Kc(Q)$. This relies on the following observation, which can be extracted from the proof of \cite[Th\,7.1]{BG2} and results from Proposition \ref{Proposition 0 bis} below.

\begin{proposition}\label{Proposition 0}
 Let $(n_{k})_{k\ge 0}$ be a strictly increasing sequence of integers with $n_{0}=1$ such that $(n_{k}\theta )_{k\ge 0}$ is uniformly distributed modulo $1$ for every $\theta \in\R\setminus D$, where $D$ is countable subset of $\R$. Then the set $Q=\{n_{k}\,;\,k\ge 0\}$ is a \ka\ subset of $\Z$ which satisfies $\Kct(Q)=\sqrt{2}$.
\end{proposition}

Consider, for every integer $p\ge 2$, the set $Q_{p}=p\,\N+1$. By Proposition \ref{Proposition 0}, $Q_{p}$ is a \ka\ subset of $\Z$ with $\Kct(Q_p)=\sqrt{2}$. But the measure $\mu =\delta _{\{e^{2i\pi /p}\}}$ satisfies
\[
\sup_{n\in Q_p}\left(1-\Re e\,\wh{\mu }(n)\right)=1-\cos(2\pi /p).
\]
Hence $\Kc(Q_p)\le\sqrt{2(1-\cos(2\pi /p))}$, which can be arbitrarily small if $p$ is sufficiently large.
\end{remark}

\section{Applications}\label{Section 6}

\subsection{Proof of Theorem \ref{Th3}}
Our first and main application of Theorem \ref{Th2} (or Corollary \ref{Cor2}) is Theorem \ref{Th3}, which solves in particular Conjecture (C4) and shows that the invariance assumption on the measure is indeed essential in the statement of Furstenberg's $\times 2$\,-$\times 3$ conjecture.

\begin{proof}[Proof of Theorem \ref{Th3}]
If $r=1$, Theorem \ref{Th3} claims the existence, for every integer $p\ge 2$, of a measure $\mu \in\mathcal{P}_{c}(\T)$ such that $\inf_{k\ge 0}|\wh{\mu }(p^{k})|>0$. As mentioned in Section \ref{sect:2}, this statement is well-known: it suffices to consider the classical Riesz product associated to the sequence $(p^{k})_{k\ge 0}$. One can also show, either as in \cite{BDLR} or \cite{EG}, or as an application of Corollary \ref{Cor1}, that $(p^{k})_{k\ge 0}$ is a rigidity sequence, so that there exists $\mu \in\mathcal{P}_{c}(\T)$ with 
 ${\wh{\mu }(p^{k})}\rightarrow{1}$ as ${k}\rightarrow{+\infty}$.
\par\smallskip
 Suppose now that $r\ge 2$, and consider, for every fixed index $1\le j\le r$, the set 
\[
C'_{j}=\{e^{2i\pi n p_{j}^{-l}}\,;\,n,l\ge 0\}
\] 
of roots of all powers of $p_j$.
It is dense in $\T$, and has the following property: there exists for every $\lambda \in C'_{j}$ an integer $l_{j}$ such that $\lambda ^{p_{1}^{k_{1}}p_{2}^{k_{2}}\dots\, p_{r}^{k_{r}}}=1$ for every $k_{j}\ge l_{j}$ and $k_{i}\ge 0$, $1\le i\le r$ with $i\neq j$. Hence
\[
\sup_{\genfrac{}{}{0pt}{}{k_{i}\ge 0}{1\le i\le r,\ i\neq j}}\bigl|\lambda ^{p_{1}^{k_{1}}\dots\, p_{r}^{k_{r}}}-1  \bigr| \rightarrow{0}\quad \textrm{as}\ {k_{j}}\rightarrow{+\infty}.
\]
Consider the two sequences $(m_{k})_{k\ge 0}$ and $(n_{k'})_{k'\ge 0}$ obtained by setting $m_{k}=p_{j}^{k}$, 
$k\ge 0$, and ordering the set 
\[
\bigl \{p_{1}^{k_{1}}\dots\, p_{j-1}^{k_{j-1}}p_{j+1}^{k_{j+1}}\dots\, p_{r}^{k_{r}}\,;\,
k_{i}\ge 0,\ 1\le i\le r\ \textrm{with}\ i\neq j\bigr\} 
\]
as a strictly increasing sequence $(n_{k'})_{k'\ge 0}$, and let ${\psi \,:\,\N}\longrightarrow{\N}$ be a strictly increasing function such that 
\[
\bigl \{p_{1}^{k_{1}}\dots\, p_{j-1}^{k_{j-1}}p_{j+1}^{k_{j+1}}\dots\, p_{r}^{k_{r}}\,;\,
0\le k_{i}\le k,\ 1\le i\le r\ \textrm{with}\ i\neq r\bigr\} 
\]
is contained in the set $\{n_{k'}\,;\,0\le k'\le \psi (k)\}$  for every $k\ge 0$. By Corollary \ref{Cor2}, there exists a measure $\mu _{j}\in\mathcal{P}_{c}(\T)$ such that 
${\wh{\mu }_{j}(p_{1}^{k_{1}}\dots\,p_{r}^{k_{r}})}\rightarrow{1}$ as ${k_{j}}\rightarrow{+\infty}$ with $0\le k_{i}\le k_{j}$, $1\le i\le r$ with $i\neq j$.
Replacing, for every $1\le j\le r$, $\mu _{j}$ by $\mu _{j}*\ti{\mu }_{j}$, we can suppose without loss of generality that $\wh{\mu }_{j}(n)\ge 0$ for every $n\in\Z$.
\par\medskip
Let now $\rho \in\mathcal{P}_{c}(\T)$ be such that $\wh{\rho} (n)>0$ for every $n\in\Z$, and set
\[
\mu =\dfrac{1}{r+1}\Bigl (\sum_{j=1}^{r}\mu _{j}+\rho  \Bigr). 
\]
Then $\mu $ is a continuous probability measure on $\T$ with $\wh{\mu }(n)>0$ for every $n\in\Z$. Moreover, we have
\begin{equation}\label{Eq13}
 \liminf\,\,\wh{\mu }\bigl ( p_{1}^{k_{1}}p_{2}^{k_{2}}\dots\, p_{r}^{k_{r}}\bigr)\ge\dfrac{1}{r+1}\quad \textrm{as}\ {\max(k_{1},\dots,k_{r})}\rightarrow{+\infty}. 
\end{equation}
Indeed, if $(k_{1}^{(l)},\dots,k_{r}^{(l)})_{l\ge 1}$ is an infinite sequence of elements of $\N^{r}$, one can extract from it a sequence (still denoted by $(k_{1}^{(l)},\dots,k_{r}^{(l)})_{l\ge 1}$) with the following property: there exists $1\le j\le r$ such that $k_{i}^{(l)}\le k_{j}^{(l)}$ for every $1\le i\le r$. Then
\[
\liminf_{{l}\rightarrow{+\infty}}\wh{\mu}\bigl (p_{1}^{k_{1}^{(l)}}\dots p_{r}^{k_{r}^{(l)}} \bigr)\ge\dfrac{1}{r+1}\liminf_{{l}\rightarrow{+\infty}}\wh{\mu}_{j}\bigl (p_{1}^{k_{1}^{(l)}}\dots p_{r}^{k_{r}^{(l)}} \bigr)=\dfrac{1}{r+1}\cdot 
\]
This yields (\ref{Eq13}). Since $\wh{\mu }(n)>0$ for every $n\ge 0$, it follows that 
\[
\inf_{\genfrac{}{}{0pt}{}{k_{i}\ge 0}{\ 1\le i\le r}}\wh{\mu }\bigl (p_{1}^{k_{1}}\dots\,p_{r}^{k_{r}} \bigr)>0,
\]
and Theorem \ref{Th3} is proved.
\end{proof}

\subsection{The case of the Furstenberg set} 
Theorem \ref{Th3} applies to the Furstenberg set $F=\{2^k3^{k'}\textrm{ ; }k,k'\ge 0\}$ and shows the existence of a measure $\mu\in\mathcal{P}_{c}(\T)$
such that \[\inf_{\genfrac{}{}{0pt}{}{k,k'\ge 0}{}}\mc(2^{k}3^{k'})>0\]
(the fact that the measure $\mu$ can be supposed to have nonnegative Fourier coefficients can be extracted from the proof of Theorem \ref{Th3}, or deduced formally from Theorem \ref{Th3} by considering the measure $\mu*\ti{\mu}$). By Corollary \ref{Cor C}, this means that $\Kct(F)<\sqrt{2}$.
\par\smallskip
As mentioned in the introduction, it is natural to look for the optimal constant $\delta \in(0,1)$ for which there exists a measure $\mu \in\mathcal{P}_{c}(\T)$ such that 
\begin{equation}\label{Eq14}
\inf_{\genfrac{}{}{0pt}{}{k,k'\ge 0}{}}\mc(2^{k}3^{k'})\ge\delta
\end{equation} 

This is equivalent to asking whether $F$ is a \ka\ set in $\Z$, and if yes, with which (modified) \ka\ constant. The best result which can be obtained via the methods presented here is that there exists a measure $\mu \in\mathcal{P}_{c}(\T)$ satisfying (\ref{Eq14}) for every
$\delta\in(0,1/2)$: this is the content of Theorem \ref{Th3bis}, which we now prove.
  \par\smallskip
\begin{proof}[Proof of Theorem \ref{Th3bis}]
 The proof goes along the same lines as that of Theorem \ref{Th3}, but it requires the full force of Theorem \ref{Th2} rather than the weaker statement of Corollary \ref{Cor2}.
\par\smallskip
Fix $\delta\in(0,1/2)$. There exist by Theorem \ref{Th2} two measures $\mu_1,\mu_2\in\mathcal{P}_{c}(\T) $ such that 
\[|\mc_{1}(2^{k}3^{k'})|\ge \sqrt{2\delta}\quad \textrm{for every } k\ge 0 \textrm{ and every }0\le k'\le k
\] and 
\[|\mc_{2}(2^{k}3^{k'})|\ge \sqrt{2\delta}\quad \textrm{for every } k'\ge 0 \textrm{ and every } 0\le k\le k'.
\]
 The measure $\mu=\frac{1}{2}(\mu_1*\ti{\mu }_1+\mu_2*\ti{\mu }_2)$ has nonnegative Fourier coefficients and satisfies
$\mc(2^{k}3^{k'})\ge \delta$ for every $k,k'\ge 0$.
\par\smallskip
It then follows from Theorem \ref{Th B} that if $\{2^k3^{k'}\textrm{ ; }k,k'\ge 0\}$ is a \ka\ subset of $\Z$, its modified \ka\ constant must be less than
$\sqrt{2(1-\delta)}$ for every $\delta \in (0,1/2)$, so must be at most $1$.
\end{proof}

That the bound $1/2$ can be further improved does not seem clear at all, and we do not know whether there exists for every $\delta \in[1/2,1)$ a measure $\mu \in\mathcal{P}_{c}(\T)$ such that \[{\inf_{\genfrac{}{}{0pt}{}{k,k'\ge 0}{}}\mc(2^{k}3^{k'})\ge \delta} .\] 

\begin{question}\label{Question 1}
 Is the Furstenberg set $\{2^{k}3^{k'}\,;\,k,k'\ge 0\}$ a \ka\ set in $\Z$?
\end{question}

Note that a lacunary semigroup $\{a^n\textrm{ ; } n\ge 0\}$, $a\ge 2$, cannot be a \ka\ set (see \cite[Ex. 5.2]{BG2}). We also observe that  Theorem \ref{Th2} immediately yields 

\begin{corollary}\label{Corollary 6}
 For any function ${\psi \,:\,\N}\rightarrow{\N}$ with  ${\psi (k)}\rightarrow{+\infty}$ as
 ${k}\rightarrow{+\infty}$,
the sets \[\{2^{k}3^{k'}\,;\,k\ge 0,\ 0\le k'\le \psi(k)\}\quad\textrm{and}\quad\{2^{k}3^{k'}\,;\,k'\ge 0,\ 0\le k\le \psi(k')\}\] are non-\ka\ sets in $\Z$. 
\end{corollary}
\par\smallskip
 Along the same lines, one can also ask for which values of $\delta \in(0,1]$ there exists a measure $\mu \in\mathcal{P}_{c}(\T)$ such that $\ul\mc(2^{k}3^{k'})\ge \delta $ as ${\max(k,k')}\rightarrow{+\infty}$. The proof of Theorem \ref{Th3} allows us to exhibit a measure $\mu \in\mathcal{P}_{c}(\T)$ with nonnegative Fourier coefficients (namely $\mu =(\mu _{1}+\mu _{2})/2$) such that $\ul\mc(2^{k}3^{k'})\ge 1/2$ as ${\max(k,k')}\rightarrow{+\infty}$. Again, we do not know whether the constant $1/2$ can be improved. The strongest statement which could be expected in this direction is the existence of a measure $\mu \in\mathcal{P}_{c}(\T)$ such that ${\mc(2^{k}3^{k'})}\rightarrow{1}$
 as 
${\max(k,k')}\rightarrow{+\infty}$. This would show that the Furstenberg sequence is a rigidity sequence for weakly mixing dynamical systems. This natural question is raised in Remark 3.12 (b) of \cite{BDLR}
Remark 3.29 c) of \cite{BDLR}, 
 and we record it anew here:

\begin{question}\label{Question 2}
 Is the Furstenberg sequence a rigidity sequence for weakly mixing dynamical systems? 
\end{question}

\subsection{Examples of rigidity sequences}
Corollaries \ref{Cor1} and \ref{Cor2} allow us to retrieve directly all known examples of rigidity sequences from \cite{BDLR}, \cite{EG}, \cite{AHL}, \cite{Aa} and \cite{FT}. The only examples of rigidity sequences not covered by our results are those of \cite{FK} and \cite{Grie}. Indeed, Fayad and Kanigowski construct in \cite{FK} examples of rigidity sequences $(\nk)_{k\ge 0}$ such that $\{\lambda ^{\nk}\,;\,k\ge 0\}$ is dense in $\T$ for every $\lambda =e^{2i\pi \theta }\in\T$ with $\theta \in\R\setminus\Q$, and there exist for every integer $p\ge 2$ infinitely many integers $k$ such that $p$ does not divide $n_k$. So such sequences never satisfy the assumption of Corollary \ref{Cor1}. Griesmer strengthens this result in \cite{Grie} by showing the existence of rigidity sequences $(\nk)_{k\ge 0}$ such that $\{n_k\,;\,k\ge 0\}$ is dense in $\Z$ in the Bohr topology.
\par\medskip
We briefly list here some of the examples of rigidity sequences which can be obtained from Corollaries \ref{Cor1} and \ref{Cor2}. Our first example is that of Fayad and Thouvenot in \cite{FT}. 

\begin{example}\cite{FT} \label{Example 1}
If the sequence $(\nk)_{k\ge 0}$ is such that there exists $\lambda =e^{2i\pi \theta }\in\T$, with $\theta \in\R\setminus \Q$, such that ${\lambda ^{\nk}}\rightarrow{1}$ as ${k}\rightarrow{+\infty}$, $(\nk)_{k\ge 0}$ is a rigidity sequence.
\end{example}

This result of \cite{FT} follows directly from Corollary \ref{Cor1}. Indeed, if ${\lambda ^{\nk}}\rightarrow{1}$ with $\lambda =e^{2i\pi \theta }$, $\theta \in\R\setminus\Q$, ${\lambda  ^{p\nk}}\rightarrow{1}$ for every $p\in\Z$. Since $\theta $ is irrational, the set $\{\lambda ^{p}\,;\,p\in\Z\}$ is dense in $\T$, and Corollary \ref{Cor1} applies.

\begin{example}\cite{BDLR}, \cite{EG} \label{Example 2}
  If $(\nk)_{k\ge 0}$ is a strictly increasing sequence of integers such that $\nk|n_{k+1}$ for every $k\ge 0$, $(\nk)_{k\ge 0}$ is a rigidity sequence. 
\end{example}

Indeed, under the assumption of Example \ref{Example 2}, the set $C=\{\lambda \in\T\,;\,{\lambda ^{\nk}}\rightarrow{1}\}$ contains all $\nk$-th roots of $1$,
$k\ge 0$, and is hence dense in $\T$.
\par\smallskip
Corollary \ref{Cor2} shows that Example \ref{Example 2} can be improved into

\begin{example} \label{Example 3}
  Let $(m_{k})_{k\ge 0}$ be a strictly increasing sequence of integers such that $m_{k}|m_{k+1}$ for every $k\ge 0$. Let 
${\psi :\N}\longrightarrow{\N}$ be a strictly increasing function. Order the set 
$\{k'm_{k}\,;\,k\ge 0\:,\; 1\le k'\le \psi (k)\}$ as a strictly increasing sequence $(\nk)_{k\ge 0}$. Then $(\nk)_{k\ge 0}$ is a rigidity sequence.
\end{example}

Indeed, the set $C'=\{\lambda \in\T\,;\,{\lambda ^{k'm_{k}}}\rightarrow{1}\,\textrm{as}\ {k}\rightarrow{+\infty}\,\textrm{uniformly in}\ k'\}$ contains all $m_{k}$-th roots of $1$, and is dense in $\T$. So Corollary \ref{Cor2} applies.
\par\smallskip
For instance, if $(r_{k})_{k\ge 0}$ is any sequence of positive integers, the sequence $(\nk)_{k\ge 0}$ obtained by ordering the set $\{k'2^{k}\,;\,k\ge 0,\;1\le k'\le r_{k} \}$ in a strictly increasing sequence is a rigidity sequence.
This provides new examples of rigidity sequences $(\nk)_{k\ge 0}$ such that ${\frac{n_{k+1}}{\nk}}\rightarrow{1}$ as ${k}\rightarrow{+\infty}$.

\begin{example}\label{Example 4}
 Let $(r_{k})_{k\ge 0}$ be any sequence of positive integers with ${r_{k}}\rightarrow{+\infty}$ as ${k}\rightarrow{+\infty}$. The sequence $(n_l)_{l\ge 0}$ obtained by ordering in a strictly increasing fashion the set $\{j2^{k}\,;\, k\ge 0,\;1\le j\le r_{k}\}$ is a rigidity sequence which satisfies ${\frac{n_{l+1}}{n_l}}\rightarrow{1}$ as ${l}\rightarrow{+\infty}$.
\end{example}

\begin{proof}
 It suffices to show that for every $\varepsilon >0$ and every $l$ sufficiently large there exists $l'>l$ such that $\frac{n_{l'}}{n_{l}}<1+\varepsilon $. 

-- Suppose first that 
 $n_l=j2^{k}$ for some $k\ge 0$ and some $1/\varepsilon <j<r_{k}$. Then taking $n_{l'}=(j+1)2^{k}$, we have $\frac{n_{l'}}{n_{l}}=\frac{j+1}{j}<1+\varepsilon $. 

-- Suppose next that $n_l'=j2^{k}$ for some $k\ge 0$ and some $1\le j\le 1/\varepsilon$. Fix an integer $p$ such that $2^{-p}<\varepsilon $.
 If $l$ is sufficiently large, we have $r_{k-p}> 2^{p}/\varepsilon $. Set $j'=j2^{p}$. Since $j'\le 2^{p}/\varepsilon <r_{k-p}$, the integer $n_{l'}=(j'+1)2^{k-p}$ appears in the sequence $(n_l)_{l\ge 0}$. Also, since $n_{l'}=(j'+1)2^{k-p}>j2^k=n_l$, we have $l'>l$, and 
 \[\frac{n_{l'}}{n_{l}}=\frac{(j'+1)2^{k-p}}{j2^{k}}=\frac{(j'+1)}{j}2^{-p}\le\frac{j+2^{-p}}{j}<1+2^{-p}<1+\varepsilon.\]

-- The last case we have to deal with is when $n_l=r_{k}2^{k}$ for some $k\ge 0$. Let $j'\ge 1$ be such that 
$j'\le r_{k}/2<j'+1$. Then $j'<r_{k+1}$, and if we set $n_{l'}=(j'+1)2^{k+1}$, the integer $n_{l'}$ appears in the sequence $(n_l)_{l\ge 0}$. We have
\[
\frac{n_{l'}}{n_{l}}=\frac{(j'+1)2^{k+1}}{r_{k}2^{k}}=\frac{2(j'+1)}{r_{k}}\le 1+\frac{2}{r_{k}}<1+\varepsilon 
\]
if $k$ is sufficiently large, and this terminates the proof.
\end{proof}

\begin{example} \cite{Aa}  \label{Example Aaronson}
(a) Let $(d_k)_{k\ge 0}$ be a strictly increasing sequence of positive integers of density zero. There exists a strictly increasing sequence of integers $(n_{k})_{k\ge 0}$ which is a rigidity sequence and satisfies $n_k \le d_k$ for every $k\ge 0$.

(b) Let $(d_k)_{k\ge 0}$ be a sequence of real numbers with $d_k\ge k$ for every $k\ge 0$ and $\frac{d_k}{k} \rightarrow +\infty$ as $k\rightarrow +\infty$. There exists a strictly increasing sequence of integers $(n_{k})_{k\ge 0}$ which is a rigidity sequence and satisfies $n_k \le d_k$ for every $k\ge 0$.
\end{example}

This has been proved by Aaronson in \cite[Th. 4]{Aa}; a simpler construction with the weaker conclusion that $n_k \le d_k$ for infinitely many $k$ was given in \cite[Prop. 3.18]{BDLR}. The proof given below uses Corollary \ref{Cor1} and a result of Bugeaud \cite{Bug2}.

\begin{proof}
As the statement (a) is a simple consequence of (b), we only give the proof of (b). Set $g_0=1$ and $g_k = d_k/k$ for every $k\ge 1$. Then $(g_{k})_{k\ge 0}$ is a sequence of reals with $ g_k\ge 1$ for every $k\ge 0$ which tends to infinity (notice that for (a) this holds since $(d_k)_{k\ge 0}$ is a sequence of density zero). Using (a particular case of) \cite[Th. 1]{Bug2}, we obtain that there exists for every fixed irrational number $\theta$ an increasing sequence $(n_{k})_{k\ge 0}$ of positive integers such that $n_k \le kg_k = d_k$ for every $k\ge 1$ and ${\exp(2i\pi\theta)^{n_k}}\rightarrow{1}$. It follows from Example \ref{Example 1} that $(n_{k})_{k\ge 0}$ is a rigidity sequence. 
\end{proof}

\begin{example}  \label{Example new}
Let $(m_k)_{k\ge 0}$ be a strictly
 increasing sequence of positive integers with ${m_{k+1}-m_k}\rightarrow{+\infty}$. There exists a strictly increasing sequence of integers $(n_{k})_{k\ge 0}$ which is a rigidity sequence and satisfies $m_k\le n_k < m_{k+1}$ for every $k\ge 0$.
\end{example}

\begin{proof}
The proof is exactly the same as the preceding one, replacing the result from \cite{Bug2} by \cite[Obs. 1.36]{BerSim}.
\end{proof}

\subsection{Exceptional sets for (almost) uniform distribution}\label{Subsection 5.3}
Let $(n_{k})_{k\ge 0}$ be a strictly increasing sequence of integers, and let $\nu \in \mathcal{M}(\T)$ be a (finite) complex Borel measure on $\T$. We stress that $\nu$ is not necessarily a probability measure. Given $\theta \in \R$, the sequence $(n_{k}\theta )_{k\ge 0}$ is said (\cite{LyonsAnnals}, \cite[p. 53]{KN}) to be \emph{almost uniformly distributed with respect to} $\nu$ 
if there exists a strictly increasing sequence $(N_j)_{j\ge1}$ of positive integers such that
 for every arc $I \subset \T$ whose endpoints are not atoms (mass-points) for $\nu$ one has 
$$ \lim_{j\to+\infty}\frac{1}{N_j} \# \left\{n \le N_j : \exp(2i\pi n_k\theta) \in I\right\} = \nu(I) .$$
The analog of Weyl's criterion states that $(n_{k}\theta )_{k\ge 0}$ is almost uniformly distributed with respect to $\nu$ if and only if 
there exists a strictly increasing sequence $(N_j)_{j\ge1}$ of positive integers such that
$$\lim_{j\to\infty}\frac{1}{N_j} \sum_{k=1}^{N_j}\exp(m2i\pi n_k\theta) \, \textrm{ exists} \quad\textrm{ for every } m\in \Z.$$
In this case, the limit is $\hat{\nu}(m)$. It can also be proved that $(n_{k}\theta )_{k\ge 0}$ is almost uniformly distributed with respect to $\nu$ if and only if 
there exists a strictly increasing sequence $(N_j)_{j\ge1}$ of positive integers such that
\[
{\fl{\dfrac{1}{N_j}\ds\sum_{k=1}^{N_j}f\bigl (e^{2i\pi n_{k}\theta } \bigr) }{\ds\int_{\T}f\,d\mu }}\quad \textrm{as}\quad {\fl{j}{+\infty}}\quad \textrm{for every}\ f\in\mathcal{C}(\T).
\]
We now denote by $W((n_{k})_{k\ge 0},\nu)$, the \emph{exceptional set of almost uniform distribution} of $(n_{k})$ with respect to $\nu$. This is the set of all 
$\theta\in\R$ such that $(n_{k}\theta )_{k\ge 0}$ is not almost uniformly distributed with respect to $\nu$. We will write $U((n_{k})_{k\ge 0},\nu)$ for the exceptional set of (classical) uniform distribution of $(n_{k})$ with respect to $\nu$, which corresponds to the case where $N_j=j$ for every $j\ge 1$. 

\par\smallskip
The size of the exceptional set $U((n_{k})_{k\ge 0},\nu )$ has been studied in  many works, in particular in the case where $\nu $ is the normalized Lebesgue measure on $\T$. In this case, we write it as $U((n_{k})_{k\ge 0})$. If the sequence $(n_{k})_{k\ge 0}$ is lacunary, $U((n_{k})_{k\ge 0})$ is uncountable, and even of Hausdorff dimension $1$ (\cite{ET}, see also \cite{HK}). See also \cite{Poll} and \cite{dM} for a stronger result. On the other hand, it is known (see \cite{Bos1}, \cite{Bou}) that among various natural classes of random sequences of integers, almost all sequences $(n_{k})_{k\ge 0}$ satisfy $U((n_{k})_{k\ge 0})=\Q$. These typical random sequences $(n_{k})_{k\ge 0}$ are sublacunary, i.e.\ satisfy ${n_{k+1}/n_{k}}\rightarrow{1}$ as ${k}\rightarrow{+\infty}$ Nonetheless, examples of sublacunary sequences $(n_{k})_{k\ge 0}$ with $U((n_{k})_{k\ge 0})$ uncountable were constructed in \cite{ET} (see also \cite{Baker}).  Concerning the size of $W((n_{k})_{k\ge 0},\nu)$ we refer for instance to \cite{Piat}, \cite{HK} and \cite{Kah}. See also \cite{Bug} for other references.
\par\smallskip

Our results about the size of $W((n_{k})_{k\ge 0},\nu)$ rely on the following generalization of Proposition \ref{Proposition 0}, which provides a link between the size of the exceptional
 set $W((n_{k})_{k\ge 0},\nu )$ and the modified \ka\ constant of the set $\{n_{k}\,;\,k\ge 0\}$.

 \begin{proposition}\label{Proposition 0 bis}
  Let $(n_{k})_{k\ge 0}$ be a strictly increasing sequence of positive integers with $n_0=1$, and let $\nu \in\mathcal{M}(\T)$ with $\nu\not =\delta_{\{1\}}$. If $W((n_{k})_{k\ge 0},\nu )$ is finite or countable infinite, $Q=\{n_{k}\,;\,k\ge 0\}$ is a \ka\ subset of $\Z$, and $$\Kct(Q)\ge\sqrt{2(1-\Re e\,\wh{\nu }(1))}.$$
 \end{proposition}

\begin{proof}
 Fix $\gamma \in(0, 1-\Re e\,\wh{\nu }(1))$, and let $\mu$ be a probability measure on $\T$ such that $\sup_{k\ge 0}(1-\Re e\,\wh{\mu }(n_{k}))<\gamma $. Then
 \begin{align*}
  1-\Re e\,\int_{\T}\Bigl (\dfrac{1}{N}\sum_{k=1}^{N}\lambda ^{n_{k}} \Bigr)d\mu (\lambda )<\gamma \quad\textrm{ for every } N\ge 1. 
 \end{align*}
Suppose that the measure $\mu$ is continuous.
Since there exists a strictly increasing sequence $(N_j)_{j\ge 1}$ of integers such that 

\[{\dfrac{1}{N_j}\ds\sum_{k=1}^{N_j}\lambda ^{n_{k}}}\rightarrow{\wh{\nu }(1)}\quad\textrm{ as } {j}\rightarrow{+\infty} \quad\textrm{ for every } \lambda \in\T\setminus C,\] where $C$ is a finite or countable infinite subset of $\T$, we have $1-\Re e\,\wh{\nu }(1)\le\gamma$, which contradicts our initial assumption. So $\mu$ has a discrete part. It then follows from Theorem \ref{Th B}  that the modified \ka\ constant of $Q$ is at least $\sqrt{2(1-\Re e\,\wh{\nu }(1))}$.
\end{proof}

The following result provides an example of a nonlacunary semigroup $(n_{k})_{k\ge 0}$ whose associated exceptional sets $W((n_{k})_{k\ge 0},\nu )$ with respect to $\nu$ are uncountable for a large class of measures $\nu \in\mathcal{M}(\T)$.

\begin{theorem}\label{Th6}
 Denote by $(n_{k})_{k\ge 0}$ the sequence obtained by ordering the Furstenberg set 
 $F=\{2^{k}3^{k'}\,;\,k,k'\ge 0\}$ in a strictly increasing fashion. For every measure $\nu \in\mathcal{M}(\T)$ such that $\Re e\,\wh{\nu }(1)<1/2$, the set 
 $W((n_{k})_{k\ge 0},\nu )$ is uncountable.
\end{theorem}

\begin{proof}[Proof of Theorem \ref{Th6}]
Fix $\nu \in\mathcal{M}(\T)$, and suppose that $U((n_{k})_{k\ge 0},\nu )$ is at most countable. Since $\Kct(F)\le 1$ by Theorem \ref{Th3bis}, it follows from Proposition \ref{Proposition 0 bis}  that $\sqrt{2(1-\Re e\,\wh{\nu }(1))}\le 1$, i.e.\ that $\Re e\,\wh{\nu }(1)\ge 1/2$. This proves Theorem \ref{Th6}.
\end{proof}

\end{document}